\documentclass[tbtags,10pt,a4paper]{amsart} 

\numberwithin{equation}{section}
\allowdisplaybreaks 

\usepackage{dsfont,stmaryrd}
\usepackage{mathtools} 

\usepackage{tikz}
\usetikzlibrary{shapes.geometric, arrows}
\usetikzlibrary{decorations.pathreplacing,calligraphy}
\usetikzlibrary{shapes.misc}
\usetikzlibrary{fit}

\usepackage{amssymb,eucal,mathrsfs,setspace} 
\usepackage[noadjust,nobreak]{cite}          
\usepackage{xcolor}
\usepackage[inner=25mm, outer=25mm, top=24mm, bottom=24mm, head=10mm, foot=10mm]{geometry}

\usepackage{array}
\newcolumntype{C}{>{$}c<{$}} 

\usepackage{subdepth}

\usepackage[largesc,theoremfont,nohelv]{newtxtext} 
\usepackage[libertine,cmbraces,varbb,smallerops]{newtxmath}

\newcommand{\umph}[1]{\textnormal{#1}}             

\begingroup
  \makeatletter
  \@for\theoremstyle:=definition,remark,plain\do{%
    \expandafter\g@addto@macro\csname th@\theoremstyle\endcsname{%
      \addtolength\thm@preskip{.5\baselineskip plus .2\baselineskip minus .2\baselineskip}
      \addtolength\thm@postskip{.5\baselineskip plus .2\baselineskip minus .2\baselineskip}
    }%
  }
\endgroup

\usepackage{enumitem}
\setitemize{leftmargin=*}         
\setenumerate{leftmargin=*,       
  label=\textnormal{\arabic*.\@}, 
	ref=\textnormal{\arabic*}       
	}

\usepackage[colorlinks=true,citecolor=red,linkcolor=blue,linktoc=all]{hyperref}

\usepackage{zref-clever}
\zcsetup{cap,nameinlink=false}                         
\newcommand{\lcnamezcref}[1]{\zcref*[noref,nocap]{#1}} 

\theoremstyle{plain}
\newtheorem{theorem}{Theorem}[section]
\newtheorem{corollary}[theorem]{Corollary}
\newtheorem{lemma}[theorem]{Lemma}
\newtheorem{proposition}[theorem]{Proposition}

\newtheorem{definition}[theorem]{Definition}
\newtheorem{remark}[theorem]{Remark}

\AddToHook{env/lemma/begin}{\zcsetup{countertype={theorem=lemma}}}
\AddToHook{env/definition/begin}{\zcsetup{countertype={theorem=definition}}}
\AddToHook{env/proposition/begin}{\zcsetup{countertype={theorem=proposition}}}
\AddToHook{env/corollary/begin}{\zcsetup{countertype={theorem=corollary}}}
\AddToHook{env/conjecture/begin}{\zcsetup{countertype={theorem=conjecture}}}
\AddToHook{env/remark/begin}{\zcsetup{countertype={theorem=remark}}}

\newcommand{\pd}{\partial}     
\newcommand{\wun}{\mathds{1}}  
\newcommand{\eps}{\varepsilon} 

\newcommand{\dd}{\mathrm{d}}   
\newcommand{\ee}{\mathsf{e}}   
\newcommand{\hh}{\mathsf{h}}   
\newcommand{\kk}{\mathsf{k}}   
\newcommand{\uu}{\mathsf{u}}   
\providecommand{\vv}{}\renewcommand{\vv}{\mathsf{v}} 

\renewcommand{\simeq}{\cong}              

\newcommand{\ses}[3]{0 \longrightarrow #1 \longrightarrow #2 \longrightarrow #3 \longrightarrow 0}

\renewcommand{\ge}{\mathrel{\geq}} 
\renewcommand{\le}{\mathrel{\leq}} 

\DeclareFontFamily{U}{mathx}{}															
\DeclareFontShape{U}{mathx}{m}{n}{<-> mathx10}{}										
\DeclareSymbolFont{mathx}{U}{mathx}{m}{n}												
\DeclareMathAccent{\widecheck}{0}{mathx}{"71}

\DeclarePairedDelimiter{\brac}{\lparen}{\rparen}   
\DeclarePairedDelimiter{\sqbrac}{\lbrack}{\rbrack} 
\DeclarePairedDelimiter{\set}{\lbrace}{\rbrace}    
\DeclarePairedDelimiterX{\setbar}[2]{\lbrace}{\rbrace}{#1 \,\delimsize\vert\,\mathopen{} #2}  
\DeclarePairedDelimiter{\ket}{\lvert}{\rangle}
\DeclarePairedDelimiter{\dket}{\lvert}{\rangle\!\rangle}     
\DeclarePairedDelimiter{\abs}{\lvert}{\rvert}

\DeclarePairedDelimiterX{\comm}[2]{\lbrack}{\rbrack}{#1 , #2}  
\DeclarePairedDelimiterX{\bilin}[2]{\langle}{\rangle}{#1 , #2} 
\newcommand{\no}[1]{\mathopen{:} #1 \mathclose{:}} 
\DeclarePairedDelimiterXPP{\killing}[2]{\kappa}{\lparen}{\rparen}{}{#1, #2} 

\DeclarePairedDelimiter{\eqclass}{\lbrack}{\rbrack} 

\DeclareMathOperator{\ad}{ad}
\DeclareMathOperator{\End}{End}
\DeclareMathOperator{\id}{id}
\DeclareMathOperator{\Aut}{Aut}
\DeclareMathOperator{\OAut}{Out}
\DeclareMathOperator{\im}{im}
\DeclareMathOperator{\spn}{span}
\newcommand{\cspn}{\spn_{\CC}}
\DeclareMathOperator{\Tr}{tr}

\newcommand{\blank}{{-}}                           

\newcommand{\fld}[1]{\mathbb{#1}}    
\newcommand{\alg}[1]{\mathfrak{#1}}  
\newcommand{\grp}[1]{\mathsf{#1}}    
\newcommand{\VOA}[1]{\mathsf{#1}}    
\newcommand{\Mod}[1]{\mathcal{#1}}   
\newcommand{\categ}[1]{\mathscr{#1}} 

\newcommand{\ZZ}{\fld{Z}}
\newcommand{\NN}{\ZZ_{\ge0}} 

\newcommand{\CC}{\fld{C}}

\newcommand{\affine}[1]{\widehat{#1}}
\newcommand{\SLA}[2]{\alg{#1}_{#2}}                      
\newcommand{\AKMA}[2]{\affine{\alg{#1}}_{#2}}            

\newcommand{\dcox}{\hh^{\vee}}				   

\newcommand{\lie}{\alg{g}}
\newcommand{\alie}{\affine{\lie}}

\newcommand{\sltwo}{\SLA{sl}{2}}
\newcommand{\asltwo}{\AKMA{sl}{2}}
\newcommand{\slthree}{\SLA{sl}{3}}

\newcommand{\rootsys}{\Delta}                           

\newcommand{\orbit}{\mathbb{O}}

\DeclarePairedDelimiterXPP{\uealg}[1]{\mathsf{U}}{\lparen}{\rparen}{}{#1}   

\newcommand{\bcsymb}{\Lambda}
\newcommand{\bgsymb}{\VOA{G}}
\newcommand{\lsymb}{\Pi}

\newcommand{\heis}{\VOA{H}}                        
\newcommand{\bgvoa}{\bgsymb}                       
\newcommand{\bcvoa}{\bcsymb}                       
\newcommand{\lvoa}{\lsymb}                         
\newcommand{\vir}{\VOA{Vir}}                       

\DeclarePairedDelimiterXPP{\uaff}[2]{\VOA{V}^{#1}}{\lparen}{\rparen}{}{#2}            
\DeclarePairedDelimiterXPP{\saff}[2]{\VOA{L}_{#1}}{\lparen}{\rparen}{}{#2}            
\DeclarePairedDelimiterXPP{\uwalg}[2]{\VOA{W}^{#1}}{\lparen}{\rparen}{}{#2}           
\DeclarePairedDelimiterXPP{\swalg}[2]{\VOA{W}_{#1}}{\lparen}{\rparen}{}{#2}           

\newcommand{\uafflie}[1][\kk]{\uaff{#1}{\lie}}
\newcommand{\usltwo}[1][\kk]{\uaff{\kk}{\sltwo}}

\newcommand{\ssltwo}[1][\kk]{\saff{\kk}{\sltwo}}
\newcommand{\sslthree}[1][\kk]{\saff{\kk}{\slthree}}

\newcommand{\sltwomm}[1][\uu,\vv]{\VOA{A}_1(#1)}

\newcommand{\uvir}[1][\kk]{\vir^{#1}}              
\newcommand{\svir}[1][\kk]{\vir_{#1}}              
\newcommand{\virmm}[1][\uu,\vv]{\vir(#1)}             

\newcommand{\uwalglie}[1][f]{\uwalg{\kk}{\lie, #1}}  
\newcommand{\swalglie}[1][f]{\swalg{\kk}{\lie, #1}}  

\newcommand{\Tvir}{T}								

\newcommand{\wtd}{\widetilde{d}}

\usepackage{cancel}
\newcommand{\omitmode}[1]{\,\cancel{\vphantom{\wtd_{-\alpha_i}} #1}\,}

\newcommand{\mbasis}[1][\Mod{M}]{\mathbf{B}_{#1}}  

\NewDocumentCommand{\wlatt}{O{\lambda}}{\lvoa_{[#1]}}
\NewDocumentCommand{\specflatt}{O{\lambda} O{\ell}}{\wlatt[#1]^{#2}} 

\newcommand{\bcvac}{\ket{0}_{\bcsymb}}

\newcommand{\cMod}{\textrm{-mod}}                 
\newcommand{\catO}{\widehat{\categ{O}}_{\kk}}			
\newcommand{\catW}{\widehat{\categ{W}}_{\kk}}			

\newcommand{\lenp}[1]{l \brac{#1}}
\newcommand{\partitions}{\mathcal{P}}

\newcommand{\partitionsFerm}{\partitions_{\text{ferm}}}
\newcommand{\multp}[2]{m_{#2} \brac{#1}}

\newcommand{\specf}{\gamma}                             
\newcommand{\specfl}[1]{\specf^{#1}}              
\newcommand{\specflvoa}[1]{\specf^{#1}_{\lvoa}}   
\newcommand{\specfsltwo}[1]{\specf^{#1}_{\sltwo}} 

\newcommand{\genvec}[2]{ {\ket{#1, #2}} }

\newcommand{\diff}{\mathrm{D}}                                   
\newcommand{\cBRST}{Q}                                           
\newcommand{\cst}{Q^{\text{st}}}                                 
\newcommand{\cchi}{Q^{\chi}}                                     

\newcommand{\Li}{\textnormal{Li}}      
\newcommand{\Lif}{\Li}                 
\newcommand{\HLif}{\widetilde{\Lif}}   
\DeclareMathOperator{\Gr}{Gr}          
\newcommand{\Ligr}{\Gr_{\Li}}          
\DeclareMathOperator{\HGr}{\Gr}        

\newcommand{\fV}{\widecheck{F}}                 

\newcommand{\functor}[1]{\mathrm{#1}} 

\newcommand{\QH}{\functor{qH}}                                  
\newcommand{\nQH}{\functor{qH}_-}                               

\NewDocumentCommand{\IH}{O{[\lambda]}}{\functor{iH}_{#1}}       

\DeclarePairedDelimiterXPP{\cohom}[2]{\mathrm{H}^{#1}}{\lparen}{\rparen}{}{#2}

\DeclareMathOperator{\Ext}{Ext}

\DeclareMathOperator{\ghgr}{\text{gr}_{\text{gh}}}
\DeclareMathOperator{\grgh}{\text{gr}_{\text{gh}}}

\newcommand{\kdelta}[2]{\delta_{#1, #2}}					

\newcommand{\dualrootlatt}{\grp{P}^{\vee}}								
\newcommand{\sltwowt}{\Lambda}
\newcommand{\sltwohwt}[1][r,s]{\sltwowt_{#1}}
\newcommand{\sltwoconfhwt}[1][r,s]{\Delta^{\sltwo}_{#1}}

\newcommand{\sltwohwm}[1][r,s]{\Mod{D}^+_{#1}}              
\newcommand{\sltwolwm}[1][r,s]{\Mod{D}^-_{#1}}              
\newcommand{\sltwoirrhwm}[1][{[\sltwowt]};r,s]{\Mod{E}_{#1}}              
\newcommand{\sltworhwmpos}[1][r,s]{\Mod{R}^+_{#1}}              
\newcommand{\sltworhwmneg}[1][r,s]{\Mod{R}^-_{#1}}              
\newcommand{\sltworhwm}[1][r,s]{\Mod{R}_{#1}}              
\newcommand{\sltwoqfin}[1][r]{\Mod{L}_{#1}}				
\newcommand{\sltwoproj}[1][r,s]{\Mod{P}_{#1}}              

\NewDocumentCommand{\virirred}{O{r, s}}{L^{\vir}_{#1}}              
\newcommand{\virwt}{\Delta}
\NewDocumentCommand{\virhwt}{O{r, s}}{\virwt_{#1}}

\newcommand{\vvoa}{\VOA{V}}							
\newcommand{\bvvoa}{\overline{\vvoa}}				
\newcommand{\scrn}[1]{\mathrm{#1}}		

\newcommand{\ppmod}{\Mod{P}}

\newcommand{\bppmod}{\overline{\ppmod}}

\makeatletter
\DeclareRobustCommand{\loplus}{\ensuremath{\mathbin{\mathpalette\loplus@{}}}}
\newcommand{\loplus@}[2]{%
  \begingroup
  \setbox0=\hbox{$#1\oplus$}%
  \dimen0=.5\wd0
  \dimen2=.08\wd0
  \dimen6=.6\dimen2
  \dimen8=.55\dimen2
  \dimen10=.35\dimen2
  \edef\loplus@R{\the\dimen0}%
  \edef\loplus@LW{\the\dimen2}%
  \edef\loplus@XL{\the\dimexpr-\dimen0-\dimen6\relax}%
  \edef\loplus@YT{\the\dimexpr \dimen0+\dimen6\relax}%
  \edef\loplus@XR{\the\dimen8}%
  \edef\loplus@HL{\the\dimexpr-\dimen0+\dimen10\relax}%
  \edef\loplus@HR{\the\dimexpr \dimen0-\dimen10\relax}%
  \vcenter{\hbox{%
    \tikz[line cap=round,line join=round]{%
      \begin{scope}
        \clip (\loplus@XL,\loplus@XL) rectangle (\loplus@XR,\loplus@YT);
        \draw[line width=\loplus@LW] (0pt,0pt) circle[radius=\loplus@R];
        \draw[line width=\loplus@LW] (-\loplus@R,0pt) -- (\loplus@R,0pt);
        \draw[line width=\loplus@LW] (0pt,-\loplus@R) -- (0pt,\loplus@R);
      \end{scope}
      \draw[line width=\loplus@LW] (\loplus@HL,0pt) -- (\loplus@HR,0pt);
    }%
  }}%
  \endgroup
}
\makeatother

\newcommand{\ds}{Drinfeld--Sokolov}

\newcommand{\km}{Kac--Moody}

\newcommand{\tkf}{the K\"{u}nneth formula}

\newcommand{\rhs}{right-hand side}

\newcommand{\hw}{highest-weight}
\newcommand{\hwv}{\hw\ vector}
\newcommand{\hwvs}{\hwv s}
\newcommand{\hwm}{\hw\ module}
\newcommand{\hwms}{\hwm s}

\newcommand{\rhw}{relaxed highest-weight}
\newcommand{\rhwv}{\rhw\ vector}
\newcommand{\rhwvs}{\rhwv s}
\newcommand{\rhwm}{\rhw\ module}
\newcommand{\rhwms}{\rhwm s}
\newcommand{\fr}{fully relaxed}
\newcommand{\frm}{\fr\ module}
\newcommand{\frms}{\frm s}
\newcommand{\chw}{conjugate-highest-weight}
\newcommand{\chwv}{\chw\ vector}

\newcommand{\vo}{vertex-operator} 
\newcommand{\voa}{\vo\ algebra}
\newcommand{\voas}{\voa s}

\newcommand{\va}{vertex algebra}
\newcommand{\vas}{\va s}

\newcommand{\emt}{energy-momentum tensor}
\newcommand{\emts}{\emt s}
\newcommand{\ope}{operator-product expansion}
\newcommand{\opes}{\ope s}
\newcommand{\qhr}{quantum hamiltonian reduction} 
\newcommand{\qhrs}{\qhr s}
\newcommand{\adjqhr}{quantum-hamiltonian-reduction} 
\newcommand{\iqhr}{inverse \qhr}

\newcommand{\walg}{W-algebra}
\newcommand{\walgs}{\walg s}

\colorlet{arrowcolour}{blue!70}
\colorlet{dotcolour}{gray!100!}
\colorlet{colourBackground1}{green!30}

\makeatletter
\renewcommand\author@andify{%
  \nxandlist {\unskip ,\penalty-1 \space\ignorespaces}%
    {\unskip {} \@@and~}%
    {\unskip \penalty-2 \space \@@and~}%
}
\makeatother

\DeclareRobustCommand{\SkipTocEntry}[5]{}

\usepackage{xpatch}
\makeatletter
\xpatchcmd{\@tocline}
          {\fHil\hbox to\@pnumwidth{\@tocpagenum{#7}}\par}
          {\ifnum#1<0\fHill\else\dotfill\fi\hbox to\@pnumwidth{\@tocpagenum{#7}}\par}
          {}{}
\makeatother

\begin{document}

\title{Reduction and inverse-reduction functors I: standard $\mathsf{V^k}(\mathfrak{sl}_2)$-modules}


\author[J~Fasquel]{Justine Fasquel}
\address[Justine Fasquel]{
	Université Bourgogne Europe\\
	CNRS, IMB UMR 5584\\
	21000 Dijon, France.
}
\email{justine.fasquel@u-bourgogne.fr}

\author[E~Fursman]{Ethan Fursman}
\address[Ethan Fursman]{
	 School of Mathematics and Statistics \\
	University of Melbourne \\
	Parkville, Australia, 3010.
}
\email{efursman@student.unimelb.edu.au}

\author[D~Ridout]{David Ridout}
\address[David Ridout]{
	School of Mathematics and Statistics \\
	University of Melbourne \\
	Parkville, Australia, 3010.
}
\email{david.ridout@unimelb.edu.au}

\begin{abstract}
	Quantum hamiltonian reduction is a fundamental tool of conformal field theory and vertex algebra representation theory.
	It has traditionally been applied to study highest-weight modules.
	On the other hand, inverse quantum hamiltonian reduction lends itself to the study of fully relaxed highest-weight modules and their spectral flows, sometimes called the standard modules.
	This is the first of several papers that study the composition of reduction and inverse-reduction functors.
	A general formalism is presented and exemplified with the simplest example, thereby computing the action of reduction on the standard modules of the affine vertex-operator algebra associated with $\mathfrak{sl}_2$.
	The appearence of unbounded spectral sequences in this formalism may be of independent interest.
\end{abstract}

\maketitle

\markleft{J~FASQUEL, E~FURSMAN AND D~RIDOUT} 

\tableofcontents

\onehalfspacing

\section{Introduction}

\subsection{Background} \label{sec:back}

Affine \voas\ are central objects in the study of two-dimensional conformal field theory and have applications to many areas of pure mathematics including representation theory, geometry, number theory, combinatorics and tensor categories.
For each universal affine \va\ $\uafflie$, where $\kk \in \CC$ is the level and $\lie$ is a simple Lie algebra, there is a finite family of universal affine \walgs\ $\uwalglie[\orbit]$, parameterised by the nilpotent orbits $\orbit$ of $\lie$.
In particular, $\uafflie$ may itself be identified with the \walg\ corresponding to the zero orbit $\orbit = \set{0}$.

Recall that the nilpotent orbits of $\lie$ admit a partial order given by inclusion of closures: $\orbit_1 \le \orbit_2$ if $\overline{\orbit}_1 \subseteq \overline{\orbit}_2$.
This induces a natural partial order on the \walgs\ associated with $\uafflie$:
\begin{equation} \label{eq:W-poset}
	\uwalglie[\orbit_2] \le \uwalglie[\orbit_1] \quad \text{whenever} \quad \orbit_1 \le \orbit_2.
\end{equation}
With this definition, $\uafflie$ is the largest \walg, while the smallest is the principal (or regular) one.

Each \walg\ $\uwalglie[\orbit]$ may be constructed from $\uafflie$ using \qhr.
This involves choosing a representative $f \in \orbit$ of the nilpotent orbit, but the result is independent of this choice (up to isomorphism).
Consequently, we often denote the resulting \walg\ by $\uwalglie = \uwalglie[\orbit]$.

Quantum hamiltonian reduction may also be applied to $\uafflie$-modules, defining a reduction functor from a suitable category of $\uafflie$-modules to the category of $\uwalglie$-modules.
There is a related notion of \iqhr\ \cite{Sem94} that also defines inverse-reduction functors \cite{AdaRea17}, but this time from $\uwalglie$- to $\uwalglie[f']$-modules.
Here, $\uwalglie[f']$ is a ``next-largest'' \walg\ after $\uwalglie$, meaning that there is no nilpotent $f'' \in \lie$ such that $\uwalglie < \uwalglie[f''] < \uwalglie[f']$.
The name comes from composing these ``atomic'' inverse-reduction functors to obtain a functor from $\uwalglie$- to $\uafflie$-modules.

The simple quotients of these \walgs, along with those of their underlying affine \voas, enjoy rich and complex representation theories with many applications in physics.
However, such applications require fixing a category of modules.
For applications to conformal field theory, an appropriate module category should satisfy a number of conditions, such as admitting duals, being tensor and being (for some meaning of the term) modular.
Historically, significant attention has been given to the category $\catO$ whose definition may be found, for instance, in \cite[Def.~A.3]{AM23}.
However, this category is rarely tensor or modular \cite{KohFus88,GabFus01,RidFus10}.

To address these issues, one may pass from category $\catO$ to the weight category $\catW$ of finitely generated weight modules (allowing the Virasoro zero mode to act non-semisimply).
This category includes the \rhwms, which are a generalisation of \hwms\ \cite{FST98,RW15,KR19}.
(In fact, the simple objects of this category are spectral flows of \rhwms\ \cite{FutCla01}, also known as standard modules \cite{CreLog13,CR13,RidVer14}.)
Such modules are generated by \rhwvs, which satisfy the same properties as \hwvs\ except that they need not be annihilated by the zero modes of the generating fields of positive weight.
It is widely believed that $\catW$ satisfies all the conditions required by conformal field theory \cite{CR12,CR13,KawAdm21,FRR24}.
But aside from rational examples, this has only been verified rigorously for the simple \voa\ $\ssltwo$ at admissible levels $\kk$ \cite{ACK23,CMY24,NHCW24,Creu24}.

Much of the recent progress for $\lie=\sltwo$ relied on brute-force methods that will be very difficult to generalise.
One promising alternative approach involves combining information obtained from reduction and inverse-reduction functors.
Roughly speaking, the complexity of the structure of the category $\catW$ increases with the ordering \eqref{eq:W-poset} on \walgs.
In particular, if $\kk$ is admissible, then there exists a nilpotent $f \in \lie$ such that the simple \walg\ $\swalglie$ has finite semisimple representation theory \cite{Ara10,AraRat19,McRRat21}.
One can then use reduction functors to classify the modules of such \emph{exceptional} \walgs\ \cite{AraMod16,AraRat19} before applying inverse reduction to obtain information about modules of the larger \walgs.
This strategy was first demonstrated for $\lie=\sltwo$ in \cite{AdaRea17}, where it was noted that inverse-reduction functors naturally construct the standard modules.
It has since been effectively employed to study $\catW$ for a variety of \walgs\ associated with higher-rank $\lie$ \cite{AKR21, FR22, ACG21, AKR23, FRR24}.

Nevertheless, much is still unknown about these reduction and inverse-reduction functors.
For instance, it is unknown to what extent reduction is an exact functor, which is important in determining how useful it can be for studying non-semisimple modules.
One natural first step towards answering important questions like this is to determine how reduction acts on the standard modules.

\subsection{Results}

The first results on the reduction of modules outside category $\catO$ appeared in the thesis \cite{Siu19}.
There, the free-field approach to reduction pioneered by Bershadsky and Ooguri \cite{BO89} was applied to general spectral flows of \hw\ $\ssltwo$-modules and to standard modules with spectral-flow index $-1$ (with our convention, which differs from that used in \cite{Siu19}).
Unfortunately, an error in some technical computations meant that the latter result was incorrect.

The reduction of general standard $\ssltwo$-modules was subsequently addressed in \cite{ACGY20}.
As with the work reported here, their method uses inverse reduction in an essential way.
However, it is unclear to what extent their technique applies to other modules or to higher ranks.
In particular, their approach relies on finding explicit coboundaries, which is unlikely to be feasible in general.
Moreover, it breaks down completely when applied to the conjugates of the reducible standard modules (the costandard modules) and thus when we replace the ``minus'' reduction functor they consider with its ``plus'' analogue.
In a sequel, we will show that our formalism can be used to compute the minus-reduction of costandard modules, equivalently the plus-reduction of the standard ones.

Finally, we mention that the action of reduction functors twisted by spectral flow has also been considered for certain \hwms\ \cite{AF19,FN23,Sim25}.
However, as the spectral flow corresponds to a current of $\uafflie$ that survives the reduction to $\uwalglie$, these results are vacuous when $\lie = \sltwo$.

To present our main result, let $\uvir = \uwalg{\kk}{\sltwo,f}$ denote the universal Virasoro \voa.
It is the principal \walg\ of $\usltwo$.
Let $\nQH$ denote the minus \qhr\ functor associated to the nilpotent $f \in \sltwo$ (defined below in \zcref{def:QHRfunctor}) and let $\nQH^p$, $p \in \ZZ$, denote its restriction to the subspace of cohomological (ghost) degree $p$.
The corresponding inverse-reduction functors $\IH^{\ell}$ are parameterised by a coset $[\lambda] \in \CC/2\ZZ$ describing $\sltwo$-weights and a spectral flow index $\ell \in \ZZ$ (see \zcref{def:IQHRfunctor,eqn:genIQHRFunctor}).
\begin{theorem}[see \zcref{thm:sltwocomposition,thm:NegativeSpecFlowCompositionsltwo,thm:PositiveSpecFlowCompositionsltwo}] \label{thm:MainResult1}
	Given $\kk \ne -2$, $[\lambda] \in \CC / 2 \ZZ$ and $\ell \in \ZZ$, the composition of minus reduction and inverse reduction acts on an arbitrary $\uvir$-module $\Mod{M}$ as follows:
	\begin{equation} \label{eqn:NegCompFunc}
    \brac[\big]{ \nQH^p \circ \IH^{\ell} } \brac{\Mod{M}} \cong \kdelta{p}{0} \kdelta{\ell}{0} \Mod{M}.
	\end{equation}
\end{theorem}
\noindent This \lcnamezcref{thm:MainResult1} is essentially equivalent to the result in \cite{ACGY20} and serves as a non-trivial check of our formalism.

Our proof of \zcref{thm:MainResult1} uses inverse reduction to carefully set up a spectral sequence that converges to the desired reduction.
This setup is completely general and thus may be applied to higher-rank examples, something that we will address in a sequel.
However, the interesting feature of this setup is that the spectral sequences employed are \emph{unbounded}, meaning that verifying their convergence is a little more subtle than usual.
(As noted in \cite{Siu19}, this unboundedness is also a feature of the sequences used without comment by Bershadsky and Ooguri in \cite{BO89}.)

It is particularly interesting to specialise \zcref{thm:MainResult1} to the category $\catW$ of weight $\ssltwo$-modules for $\kk$ admissible.
We recall that for $\lie=\sltwo$, the level $\kk$ is admissible if
\begin{equation}\label{eqn:admissible_condition}
	\kk=-2+\frac{\uu}{\vv}, \quad \text{where}\ \uu \in \ZZ_{\ge2},\ \vv \in \ZZ_{\ge1}\ \text{and}\ \gcd\set{\uu,\vv}=1.
\end{equation}
For $\vv=1$, $\ssltwo$ is strongly rational, hence $\catW$ is finite and semisimple.
Assume therefore that $\vv>1$.
Then, the following modules belong to $\catW$ \cite{AM95}:
\begin{itemize}
	\item The simple fully relaxed \hwms\ $\sltwoirrhwm$, where $[\Lambda] \in \CC/2\ZZ$, $r=1,\dots,\uu-1$, $s=1,\dots,\vv-1$ and $[\Lambda] \ne \sqbrac[\big]{\pm\Lambda_{r,s}}$.
	Here, we set $\Lambda_{r,s} = r-1-\frac{\uu}{\vv}s$ for brevity.
	\item The non-semisimple fully relaxed \hwms\ $\sltworhwmpos$ and $\sltworhwmneg$, where $r=1,\dots,\uu-1$ and $s=1,\dots,\vv-1$.
	\item The images of these modules under twisting by the spectral flow functors $\specfsltwo{\ell}$, where $\ell \in \ZZ$, see \zcref{sec:specflows}.
\end{itemize}
Here, we employ the parameterisation of \cite{CR13} and refer to \cite{RW15,KR19} for further information about these modules and their structures.
They are all distinct except for the isomorphisms $\specfsltwo{\ell}(\sltwoirrhwm) \cong \specfsltwo{\ell}(\sltwoirrhwm[{[\Lambda];\uu-r,\vv-s}])$.

A seminal observation of \cite{AdaRea17}, see also \cite{AKR21}, is that $\specfsltwo{\ell}(\sltwoirrhwm)$ and $\specfsltwo{\ell}(\sltworhwmneg)$ lie in the image of the inverse-reduction functors $\IH[{[\Lambda/2]}]^{\ell}$ and $\IH[{[\Lambda_{r,s}/2]}]^{\ell}$, respectively.
These modules are the standard modules of $\catW$, while the $\specfsltwo{\ell}(\sltwoirrhwm)$ and $\specfsltwo{\ell}(\sltworhwmpos)$ are the costandard ones.
Consequently, the minus reductions of the standard modules are given by \zcref{thm:MainResult1}.
\begin{corollary}[see \zcref{cor:NegRedFullyRelaxedWtCat}] \label{cor:MainCor1}
	Given $\kk$ admissible (and $\vv\ne1$), the minus reductions of the standard modules of $\catW$ are
	\begin{equation}
		\nQH^p \brac*{ \specfsltwo{\ell} \brac{\sltwoirrhwm} } \cong
		\nQH^p \brac*{ \specfsltwo{\ell} \brac{\sltworhwmneg } } \cong \kdelta{p}{0} \kdelta{\ell}{0} \virirred.
	\end{equation}
\end{corollary}
\noindent Here, for each $r=1,\dots,\uu-1$ and $s=1,\dots,\vv-1$, $\virirred$ denotes the simple \hwm\ of the simple Virasoro \voa\ $\svir = \swalg{\kk}{\sltwo,f}$ whose images under $\IH[{[\Lambda/2]}]^{\ell}$ and $\IH[{[\Lambda_{r,s}/2]}]^{\ell}$ are $\specfsltwo{\ell}(\sltwoirrhwm)$ and $\specfsltwo{\ell}(\sltworhwmneg)$, respectively.

To illustrate the utility of the formalism developed here, we also apply it to the projective covers $\sltwoproj$ of the \hwms\ (and their spectral flows).
Using the explicit construction of \cite{AdaRea17}, we obtain the following result from \zcref{thm:MainResult1}.
\begin{corollary}[see \zcref{thm:QHRProjCoversPos}] \label{cor:MainCor2}
  Given $\kk$ admissible (and $\vv\ne1$), $\ell \in \ZZ$, $1 \le r \le \uu-1$ and $1 \le s \le \vv-2$, the minus reductions of the non-semisimple projectives of $\catW$ are given by
  \begin{equation}
    \nQH^p \brac*{ \specfsltwo{\ell} \brac{\sltwoproj} } \cong
    \begin{cases*}
      \kdelta{n}{0} \virirred & for $\ell=0$, \\
      \kdelta{n}{0} \virirred[r,s+1] & for $\ell=-1$, \\
      0 & otherwise.
    \end{cases*}
  \end{equation}
\end{corollary}

There are of course many additional indecomposable $\ssltwo$-modules in $\catW$.
In particular, the minus reductions of the reducible costandard modules $\specfsltwo{\ell}(\sltworhwmpos)$ are presently unknown.
We will compute them in a sequel by combining our computational formalism with a non-trivial change-of-basis theorem.
An independent sequel will address $\sslthree$-modules and make contact with the partial \qhr\ functors of \cite{MR97} that have been recently studied in, for example, \cite{FFFN24,GenRed25,FasQua25}.

\subsection{Structure of the paper}

The paper is organised as follows.
In \zcref{sec:preliminaries}, we introduce the fundamental objects of interest in this paper, starting with \zcref{subsec:QHR} where we review affine \vas, the bosonic and fermionic ghost systems, and \qhr.
In \zcref{subsec:IQHR}, we discuss the half-lattice \va\ and \iqhr.
Spectral flow automorphisms are covered in \zcref{subsec:automorphisms}, using affine \vas\ as an example.

\zcref{sec:CompositionTypeARankOne} commences our study of the affine \va\ $\usltwo$ and its principal \walg\ $\uvir$.
We first review in \zcref{sec:ReviewRankOne} the precise definitions of the reduction and inverse-reduction functors relating their modules.
This is followed by a careful analysis of the action of their composition when there is no spectral flow (the $\ell=0$ part of \zcref{thm:MainResult1}).
We introduce a general strategy for this analysis in \zcref{subsec:QHR_relaxed}, not only to clarify the proof but also with a view to applying it to higher-rank generalisations.
The main difficulty is that the filtered spectral sequence we encounter is (unavoidably) unbounded.
Nevertheless, we are able to establish that it converges to the required cohomology.

After reviewing some necessary specifics concerning spectral flow in \zcref{sec:specflows}, we compute the reduction of \frms\ with positive spectral flow (\zcref{subsec:SpecFlowTypeARank1Pos}).
This is separated from the negative spectral flow case (\zcref{subsec:SpecFlowTypeARank1Neg}) because the latter requires an additional trick to circumvent the fact that the differential of our complex now fails to annihilate the vacua of our modules.
The results of these \lcnamezcref{subsec:SpecFlowTypeARank1Pos,subsec:SpecFlowTypeARank1Neg} complete the proof of \zcref{thm:MainResult1}.

In \zcref{sec:AdmissLevel} we consider applications to the category of weight $\ssltwo$-modules at admissible level.
The reduction of the standard modules is given as an immediate corollary of \zcref{thm:MainResult1}.
This is followed by a previously unknown result: the reduction of the projective covers of the \hw\ $\ssltwo$-modules (and their spectral flows).

Spectral sequences and their convergence are reviewed in \zcref{sec:UnboundedSpecSequences}.
This includes some useful, but perhaps less well known, results regarding the convergence of unbounded spectral sequences.

\subsection{Notation} \label{subsec:partitions}

Suppose that $\VOA{V}$ is a \voa\ and that the state $A \in \VOA{V}$ corresponds to a homogeneous field given by the vertex operator $Y \brac{A, z} = A(z)$.
We denote its conformal weight by $\Delta_A$ and define the modes $A_n$ and $A_{(n)}$ by the expansions
\begin{equation}
	A(z) = \sum_{n \in \ZZ-\Delta_A} A_n z^{-n- \Delta_A} = \sum_{n \in \ZZ} A_{(n)} z^{-n-1}.
\end{equation}
We will almost exclusively work with the modes $A_n$ in what follows.

It is convenient to describe elements of the mode algebra of a \voa\ using partitions.
We denote the set of partitions of the non-negative integers by $\partitions$.
Recall that a partition $\alpha = \brac{\alpha_1, \alpha_2, \dots, \alpha_L}$ is a finite sequence of positive integers $\alpha_i$, called the parts of $\alpha$, satisfying
\begin{equation}\label{eqn:partitionOrdering}
	\alpha_1 \ge \alpha_2 \ge \dots \ge \alpha_L > 0,
\end{equation}
where $L=\lenp{\alpha}$ is the length of the partition $\alpha$.
The weight of a partition is defined to be $\abs{\alpha} = \alpha_1 + \alpha_2 + \dots + \alpha_L$.
For $N \in \NN$, the set of all partitions with weight $N$ is denoted by $\partitions \brac{N}$ and $\partitions \brac{0} = \set{\varnothing}$.
The multiplicity of $n \in \NN$ in the partition $\alpha$ is defined to be the number of times $n$ appears as a part of $\alpha$.
We denote it by
\begin{equation} \label{eqn:define-multp}
	\multp{\alpha}{n} = \abs[\big]{ \setbar{\alpha_i}{\alpha_i\ \text{is a part of}\ \alpha\ \text{and}\ \alpha_i = n} }.
\end{equation}

Given a (homogeneous) element $A \in \VOA{V}$, we introduce the following notations:
\begin{equation}
	\begin{aligned}
		A_{+ \alpha} &= A_{\alpha_L} \cdots A_{\alpha_2} A_{\alpha_1}, &&&
		A_{- \alpha} &= A_{-\alpha_1} A_{-\alpha_2} \cdots A_{-\alpha_L},
	\end{aligned}
\end{equation}
Often we wish to shift a partition $\alpha \in \partitions$ by some $M \in \ZZ$.
In this case, we define
\begin{equation}
	\alpha + M = \brac{\alpha_1+M, \alpha_2+M, \dots, \alpha_L+M}.
\end{equation}
This makes sense, even when the result is no longer a partition of any non-negative integer.

In this paper, we will also work with fermionic fields whose modes square to zero.
Consequently, we introduce the set of partitions
\begin{equation}
	\partitionsFerm = \setbar[\big]{\alpha \in \partitions}{\multp{\alpha}{n} \le 1\ \text{for all}\ n \in \NN}
\end{equation}
whose parts have multiplicity either $0$ or $1$.


\addtocontents{toc}{\SkipTocEntry} 
\subsection*{Acknowledgements}

JF's research was supported by a University of Melbourne Establishment Grant, the MATRIX-Simons Young Scholar Program and an Andrew Sisson Support Package 2025.
EF's research is supported by an Australian Government Research Training Program (RTP) Scholarship.
DR's research is supported by the Australian Research Council Discovery Project DP210101502 and an Australian Research Council Future Fellowship FT200100431.

\section{Preliminaries}\label{sec:preliminaries}

\subsection{Quantum hamiltonian reduction}\label{subsec:QHR}

Let $\lie$ be a complex simple Lie algebra with dual Coxeter number $\dcox$ and let $\kappa$ denote the (renormalised)
Killing form $\kappa \brac{x, y} = \frac{1}{2 \dcox} \Tr \brac{\ad_x \ad_y}$.
The affine \km\ algebra is
\begin{equation}
  \alie = \lie \otimes \CC[t^{\pm1}] \oplus \CC K,
\end{equation}
where $K$ is central.
Writing $x_m = x \otimes t^m$, for $x \in \lie$, the Lie algebraic structure on $\alie$ is then given by
\begin{equation}
  \comm{x_m}{y_n}=\comm{x}{y}_{m+n} + m \killing{x}{y} \kdelta{m+n}{0} K, \quad \text{for}\ x,y \in \lie\ \text{and}\ m,n \in \ZZ.
\end{equation}

For $\kk \in \CC$, the level-$\kk$ universal affine \va\ $\uafflie$ is the $\alie$-module
\begin{equation}
  \uafflie = \uealg{\alie} \otimes_{\uealg[big]{\lie[t] \oplus \CC K}} \CC_{\kk},
\end{equation}
where $\CC_{\kk}$ is the one-dimensional representation of $\uealg[\big]{\lie[t] \oplus \CC K}$ that is annihilated by $\lie[t]$ and on which $K$ acts as multiplication by the level $\kk$.
There is a \va\ structure on $\uafflie$ that is freely generated by the fields
\begin{equation}
  x(z) = \sum_{n \in \ZZ} x_n z^{-n-1}, \quad \text{for} \ x \in \lie,
\end{equation}
whose \opes\ take the form
\begin{equation}
  x(z) y(w) \sim \frac{\comm{x}{y}(w)}{(z-w)} + \frac{\killing{x}{y} \kk \wun}{(z-w)^2}, \quad \text{for} \ x, y \in \lie.
\end{equation}
We recall that strong generation means that the vertex algebra is spanned by the normally ordered products of the derivatives of the fields.
Free generation just means that all relations are derived, again using normally ordered products and derivatives, from the \opes\ of the strong generators.

For a non-critical level $\kk \ne -\dcox$, the Sugawara construction ensures that $\uafflie$ admits an \emt\ $T^{\lie}(z)$.
This may be deformed by an element of the (chosen) Cartan subalgebra $\alg{h}$ of $\lie$:
\begin{equation}\label{eqn:affineEMT}
	T^{\lie,h}(z)
	= T^{\lie}(z) + \pd h(z), \quad \text{for}\ h \in \alg{h}.
\end{equation}

Affine \walgs\ are the \vas\ obtained from an affine \va\ via a variant of BRST cohomology known as \qhr.
This was first considered in a number of early works, for example \cite{BelKdV89, BO89, Pol90, Ber91}, and then generalised in \cite{FF90, KRW03}.
To define this reduction, we fix a non-zero nilpotent orbit $\orbit$ of $\lie$, with representative $f \in \orbit$, and a good $\frac{1}{2} \ZZ$-grading $\Gamma$ for $f$.
It is known that the resulting \walg\ is independent of $f$ and $\Gamma$ \cite[\S9.2.1]{AM23}.
In this paper, we will work solely with type A, in which case a good $\ZZ$-grading can always be found \cite{BG06,EK05}.
We will therefore assume that $\Gamma$ is a good $\ZZ$-grading.

Choosing $f \in \orbit$, we define the \ds\ character $\chi \colon \alg{g} \to \CC$ by $\chi(x) = \kappa \brac{f, x}$.
Let $\rootsys$ be the root system of $\lie$ and denote by $\rootsys_{>0}$ the set of roots $\alpha$ satisfying $\Gamma(\alpha)>0$.
Quantum hamiltonian reduction involves imposing the following constraints on $\uafflie$:
\begin{equation}\label{eqn:QHRConstraints}
  J^{\alpha}(z) + \chi( J^{\alpha} ) \wun = 0, \quad \text{for}\ \alpha \in \rootsys_{>0}.
\end{equation}
Here, $J^{\alpha}$ denotes the root vector of $\lie$ corresponding to $\alpha \in \rootsys$.

To impose these constraints, we first introduce the fermionic ghost system $\bcvoa$ \cite{PolFer81}, which is freely generated by fermionic fields $\varphi(z)$ and $\varphi^*(z)$ whose only non-regular \ope\ is
\begin{align}
  \varphi(z) \varphi^*(w) \sim \frac{\wun}{z-w}.
\end{align}
This \va\ admits a one-parameter family of \emts\ given by
\begin{equation}\label{eqn:bcVOAConformalStructure}
  T^{\bcvoa,\theta}(z) = \brac{1-\theta} \no{\pd \varphi(z) \varphi^*(z)} - \theta \no{\varphi(z) \pd \varphi^*(z)}, \quad \text{for}\ \theta \in \CC.
\end{equation}
The conformal weights of the generating fields are then $\Delta_{\varphi} = \theta$ and $\Delta_{\varphi^*} = 1-\theta$.

Suppressing the dependence on $f$ and $\Gamma$ in our notation, the BRST complex is then defined by
\begin{equation}
  C^{\bullet} \brac[\big]{ \uafflie } = \uafflie \otimes \bcvoa^{\bullet}_{\Gamma}, \quad \text{where}\ \bcvoa^{\bullet}_{\Gamma} = \bigotimes_{\mathclap{\alpha \in \rootsys_{>0}}} \bcvoa^{\bullet}_{\alpha}
\end{equation}
and each $\bcvoa_{\alpha}$ is a copy of the fermionic ghost \va\ (with generators $\varphi^{\alpha}(z)$ and $\varphi^{\alpha,*}(z)$).
We equip the BRST complex $C^{\bullet} \brac[\big]{ \uafflie }$ with the ghost grading $\grgh$, so that
\begin{equation}
	\grgh \varphi_n = -1, \quad \grgh \varphi^*_n = 1, \quad \grgh x_n = 0 \quad \text{for all}\ x \in \lie, \quad \text{and} \quad \grgh \wun = 0.
\end{equation}
Lastly, we introduce the BRST current $\cBRST(z) = \cst(z) + \cchi(z)$, where
\begin{equation}\label{eqn:generalBRSTdiff}
  \cst = \sum_{\alpha \in \rootsys_{>0}} J^{\alpha} \varphi^{\alpha, *} - \frac{1}{2} \sum_{\alpha, \beta, \gamma \in \rootsys_{>0}} c^{\alpha, \beta}_{\gamma} \no{ \varphi^{\alpha, *} \varphi^{\beta,*} \varphi^{\gamma}}
  \quad \text{and} \quad
  \cchi =\sum_{\alpha \in \rootsys_{>0}} \chi ( J^{\alpha} ) \varphi^{\alpha, *}.
\end{equation}
Here, the $c^{\alpha, \beta}_{\gamma}$ are the structure constants of the nilpotent subalgebra of $\lie$ spanned by the $J^{\alpha}$ with $\alpha \in \rootsys_{>0}$, meaning that $\comm{J^{\alpha}}{J^{\beta}} = \sum_{\gamma \in \rootsys_{>0}} c^{\alpha, \beta}_\gamma J^{\gamma}$.

It follows that $\brac[\big]{C^{\bullet} (\uafflie), \diff}$ is a differential complex, where the BRST differential is $\diff = \oint_0 Q(z) \,\dd z$.
The cohomology of this complex is concentrated in ghost grade zero \cite{KW04} and the corresponding affine \walg\ is defined to be
\begin{equation}
  \uwalglie = \cohom[\big]{0}{ C\brac[\big]{ \uafflie } , \diff }.
\end{equation}

The BRST complex inherits a family of \emts\ coming from those of $\uafflie$ and $\bcvoa_{\Gamma}$:
\begin{equation}
	\widetilde{T}(z) = T^{\lie,h}(z) + \sum_{\alpha \in \rootsys_{>0}} T^{\bcvoa_{\alpha},\theta_{\alpha}}(z), \quad \text{where}\ h \in \alg{h}\ \text{and}\ \theta_{\alpha} \in \CC.
\end{equation}
To determine an appropriate conformal structure on the \walg , we require that the constraints in \eqref{eqn:QHRConstraints} are homogeneous and that the BRST differential $\diff$ is homogeneous of degree zero.
This ensures that $\diff$ leaves conformal weights invariant, hence that the appropriate $\widetilde{T}(z)$ has a well-defined action on the zeroth cohomology.
It is straightforward to check that this holds if and only if $\cBRST(z)$ has conformal weight $\Delta_{\cBRST} = 1$.
Requiring both homogeneity of constraints and $\Delta_{\cBRST} = 1$ gives a system of equations for $h$ and the $\theta_{\alpha}$ that greatly constrains $\widetilde{T}(z)$.

Reduction may also be defined for $\uafflie$-modules.
Given a $\uafflie$-module $\Mod{M}$, define the BRST complex
\begin{equation}
  C^{\bullet}( \Mod{M} ) = \Mod{M} \otimes \bcvoa^{\bullet}_{\Gamma},
\end{equation}
equipped with the BRST differential defined by the action of the zero mode of the same BRST current $\cBRST(z)$.
Often, we will simply write $C$, rather than $C \brac[\big]{\uafflie}$ or $C \brac{\Mod{M}}$, for the BRST complex when the \va\ or module being considered is clear.

Write $Y \colon C \brac[\big]{\uafflie} \to \End C(\Mod{M}) \llbracket z^{\pm 1} \rrbracket$ for the vertex operator map of the BRST complex of $\Mod{M}$.
It is well known that the cohomology $\cohom[\big]{0}{C(\Mod{M}), \diff}$ inherits the structure of a $\uwalglie$-module, for instance see \cite[Section~6.4]{AraRepW07}.
This is given by $Y_{\QH} \colon \cohom[\big]{0}{C \brac[\big]{\uafflie}, \diff} \to \End \cohom[\big]{0}{C(\Mod{M}), \diff} \llbracket z^{\pm 1} \rrbracket$ where
\begin{equation}\label{eqn:cohomModStructure}
  Y_{\QH} \brac[\big]{\eqclass{v}, z} \eqclass{m} = \eqclass[\big]{Y \brac{v, z} m},
  \quad \text{for}\ v \in \ker \diff \subseteq C \brac[\big]{\uafflie}\ \text{and}\  m \in \ker \diff \subseteq C(\Mod{M}).
\end{equation}

\begin{remark} \label{rem:NoQuasiIsoComplexes}
  In the BRST cohomology construction of \walgs, we specify a BRST current $\cBRST(z)$.
  The action of the zero mode on $C \brac{\Mod{M}}$ depends on the module structure, and this defines the BRST differential.
  Once the differential has been determined, $\brac[\big]{C \brac{\Mod{M}}, \diff}$ can be treated as a $\diff$-complex of graded vector spaces, rather than as a complex of modules.
  Once the cohomology has been computed as a graded vector space, then \eqref{eqn:cohomModStructure} equips it with the structure of a $\uwalglie$-module.
  For this reason, it is not necessary to maintain quasi-isomorphisms of modules throughout the computation of the BRST cohomology.
\end{remark}

\begin{definition}\label{def:QHRfunctor}
	The \adjqhr\ functor $\QH^p\colon \uafflie\cMod \to \uwalglie\cMod$ is the functor satisfying $\QH^p(\Mod{M}) = \cohom[\big]{p}{C(\Mod{M}), \diff}$.
\end{definition}
\noindent We will often refer to the functor $\QH^p$ as a \emph{reduction functor}, for simplicity.

Because the BRST cohomology of a module $\Mod{M}$ is often concentrated in degree $p=0$, as it is for vertex algebras, we will often drop the superscript $p$ from reduction functors.
This concentration has been proven for various specific types of modules and certain nilpotent orbits, see for instance \cite{KW04, AraVan04}.
However, there are cases in which the cohomology is known not to be concentrated in degree zero \cite{Siu19,AF19,FN23}.

\subsection{Inverse quantum hamiltonian reduction}\label{subsec:IQHR}

The notion of \iqhr\ was first introduced in \cite{SemMFF93, Sem94} for $\usltwo$ and has since been explored for various low rank cases \cite{AdaRea17,AKR21, ACG21, AKR23, CFLN24, FRR24, FFFN24, FehPri25, CreSL225} and some general families of \walgs\ \cite{FehSub23, FehHook23, FN23}.
In \cite{AdaRea17}, it was shown that inverse reduction could be used to study not only representations of the universal affine \va\ $\usltwo$ but also those of its simple quotient $\ssltwo$.
General techniques for using inverse reduction to classify the representation theory of simple quotients of \walgs\ were subsequently developed in \cite{AKR21, AKR23}.

Inverse \qhr\ usually involves the bosonic ghost system $\bgvoa$ \cite{PolBos81} and the half-lattice \va\ $\lvoa$ \cite{FMS86}.
However, the application that will concern us in this paper --- the connection between the representation theories of $\usltwo$ and $\uvir$ --- only requires $\lvoa$.

To construct this \va, consider the abelian Lie algebra $\pi = \cspn \set{ c, d }$, equipped with a symmetric bilinear form $\bilin{\blank}{\blank}$ satisfying
\begin{equation}
  \bilin{c}{c} = \bilin{d}{d}=0 \quad \text{and} \quad \bilin{c}{d}=2.
\end{equation}
Next, we take its affinisation $\affine{\pi}$ as in \zcref{subsec:QHR} (with $\bilin{\blank}{\blank}$ substituting for the Killing form).
The induced vacuum module, on which the central element $K$ acts as the identity, may be equipped with the structure of a rank-$2$ Heisenberg \va\ $\heis$ whose \opes\ are
\begin{equation}
  x(z)y(w) \sim \frac{\bilin{x}{y}\wun}{(z-w)^2}, \quad \text{for} \ x, y \in \pi.
\end{equation}

The group algebra $\CC[\ZZ c] = \cspn \setbar{e^{nc}}{n \in \ZZ}$ is equipped with the structure of a $\pi$-module through the action
\begin{equation} \label{eq:actonexp}
  x \brac*{ e^{nc} } = n \bilin{x}{c} e^{nc},\quad x\in\pi.
\end{equation}
The half-lattice vertex algebra is then given by $\lvoa = \heis \otimes \CC[\ZZ c]$, where the action of $x \in \pi$ on $\CC[\ZZ c]$ is identified with the action of the zero mode $x_0$ of the field $x(z)$ of $\heis$.
In particular, $\lvoa$ is strongly (but not freely) generated by the fields $c(z)$, $d(z)$ and $e^{nc}(z)$, for $n \in \ZZ$, whose non-trivial \opes\ are
\begin{equation}
  c(z)d(w) \sim \frac{2 \wun}{(z-w)^2} \quad \text{and} \quad d(z)e^{nc}(w) \sim \frac{2n e^{nc}(w)}{(z-w)}.
\end{equation}

There is a two-parameter family of \emts\ in $\lvoa$ given by
\begin{equation} \label{eq:genlatemt}
  T^{\lvoa,\zeta,\eta} = \frac{1}{2} \no{c(z)d(z)} + \zeta \pd c(z) + \eta \pd d(z),
  \quad \text{for} \ \zeta, \eta \in \CC.
\end{equation}
The central charge of $T^{\lvoa,\zeta,\eta}$ is $c^{\lvoa} = 2 - 48 \zeta \eta$ and the conformal weights of the strongly generating fields are $\Delta_c = \Delta_d = 1$ and $\Delta_{e^{nc}} = -2n\eta$.

The representation theory of $\lvoa$ was studied in \cite{BDT01}.
There, it was shown that the simple weight $\lvoa$-modules all have the form $\lvoa \cdot \ee^{Ac+Bd}$, with $A \in \CC$ and $B \in \eta+\frac{1}{2}\ZZ$.
(This does not define a module for all $B \in \CC$, because $e^c(z)$ need not have integer moding.)
We will discuss these simples further once we have fixed $A$ and $B$ in \zcref{sec:ReviewRankOne}.

Roughly speaking, \qhr\ involves gauging the fields associated with certain root directions of the underlying Lie algebra $\lie$.
Conversely, \iqhr\ aims to reconstruct these root directions and their associated fields using the \vas\ $\bgvoa$ and/or $\lvoa$.
These root directions are actually encoded by the choice of a nilpotent orbit $\orbit$ of $\lie$.
Fixing a representative $f \in \orbit$, the key to inverse \qhr\ is the construction of a so-called \emph{Semikhatov embedding}
\begin{equation} \label{eqn:iqhrembedding}
  \Phi\colon \uafflie \hookrightarrow \uwalglie \otimes \bgvoa^{\otimes M} \otimes \lvoa^{\otimes N},
\end{equation}
where $M,N \in \NN$ are determined by the nilpotent orbit $\orbit$.
More generally, Semikhatov embeddings like \eqref{eqn:iqhrembedding} can also be constructed with $\uafflie$ replaced by another of its \walgs\ \cite{AKR21}.

As was pointed out in \cite{AdaRea17}, inverse reduction also defines \emph{Adamovi\'{c} functors} from the modules of $\uwalglie$ to those of $\uafflie$.
Suppose that $\Mod{M}$ is a $\uwalglie$-module and that $\Mod{N}$ is a $(\bgvoa^{\otimes M} \otimes \lvoa^{\otimes N})$-module.
Using the embedding \eqref{eqn:iqhrembedding}, we can then restrict the action of $\uwalglie \otimes \bgvoa^{\otimes M} \otimes \lvoa^{\otimes N}$ on $\Mod{M} \otimes \Mod{N}$ to an action of $\uafflie$.
\begin{definition}\label{def:IQHRfunctor}
	For a given $(\bgvoa^{\otimes M} \otimes \lvoa^{\otimes N})$-module $\Mod{N}$ and Semikhatov embedding \eqref{eqn:iqhrembedding}, the Adamovi\'{c} functor $\IH[\Mod{N}] \colon \uwalglie\cMod \to \uafflie\cMod$ is defined so that $\IH[\Mod{N}](\Mod{M})$ is the restriction of $\Mod{M} \otimes \Mod{N}$.
\end{definition}
\noindent We will often refer to Adamovi\'{c} functors as \emph{inverse-reduction} functors, to highlight their relationship with the reduction functors $\QH$ of \zcref{def:QHRfunctor}.
That said, these functors are not inverses in the usual sense.

\subsection{Spectral flow automorphisms}\label{subsec:automorphisms}
Spectral flow automorphisms are an important representation-theoretic tool for constructing and studying \va\ modules.
These automorphisms were first formally introduced for \voas\ by Li in \cite{Li97} although they were already very familiar to physicists, for instance see \cite{SS87}.
Here, we briefly review spectral flow, first for affine \voas\ following \cite[App.~A]{Rid09}, and then more generally.

Given an affine Lie algebra $\alie$ with a fixed choice of Cartan subalgebra $\widehat{\alg{h}}$, we consider the automorphisms $\specf \colon \alie \to \alie$ that preserve the Cartan subalgebra and fix the central element $K$, that is $\specf \brac{\affine{\alg{h}}} \subseteq \affine{\alg{h}}$ and $\specf \brac{K} = K$.
These automorphisms form a group
\begin{equation}
	\Aut_{\alg{h}} \alie \cong \OAut \alie \ltimes \widehat{W} \cong \Aut_{\alg{h}} \lie \ltimes \dualrootlatt,
\end{equation}
where $\OAut \alie$ is the outer automorphism group of $\alie$, $\widehat{W}$ is the affine Weyl group and $\dualrootlatt \subset \alg{h}$ is the coweight lattice of $\lie$ (the integer dual of the root lattice).
The automorphisms corresponding to translations along directions of $\dualrootlatt$ are known as \emph{spectral flow automorphisms}.
The Virasoro operator $L_0$ of an affine Lie algebra is typically not invariant under the action of a spectral flow automorphism (whence the name).

When $\lie$ is a complex simple Lie algebra, the spectral flow automorphisms for $\alie$ may be lifted to the fields of the affine \va\ $\uafflie$ in the obvious way:
\begin{equation} \label{eq:affine-spec-flow}
	\specf \brac[\big]{x(z)} = \sum_{n \in \ZZ} \specf \brac{x_n} z^{-n-1}, \quad x \in \lie.
\end{equation}
These maps are generalised in Li's construction \cite{Li97} for an arbitrary \voa\ $\VOA{V}$.
For this, we need a Virasoro-primary field $j(z)$ of conformal weight $1$ satisfying
\begin{equation}
	j(z) j(w) \sim \frac{C \wun}{\brac{z-w}^2}, \quad \text{for some}\ C \in \CC,
\end{equation}
and that its zero mode $j_0$ acts semisimply on $\VOA{V}$ with integer eigenvalues.
Given such a field $j(z)$, we form the operator
\begin{subequations} \label{eq:specflow}
  \begin{equation}\label{eqn:LiSpecFlow}
    \Delta \brac{j,z} = z^{-j_0} \prod_{n=1}^{\infty} \exp \brac*{ \frac{(-z)^{-n}}{n} j_n }
  \end{equation}
  and define the spectral flow map associated with $j$ by
  \begin{equation}\label{eqn:LiSpecFlowAction}
    \specf \brac[\big]{ A(z) } = Y \brac[\big]{ \Delta \brac*{ j,z } A, z}, \quad \text{for}\ A \in \VOA{V}.
  \end{equation}
\end{subequations}
Here, $Y$ denotes the state-field correspondence of $\VOA{V}$, that is $Y(A,z) = A(z)$.
It is easy to check that when $\VOA{V} = \uafflie$, the conditions on $j$ mean that it lies in $\dualrootlatt$ and this definition recovers \eqref{eq:affine-spec-flow} with translation $\gamma$ parameterised by $j$.

Given a $\VOA{V}$-module $\Mod{M}$ and an automorphism $\gamma$ of $\VOA{V}$, we define the $\VOA{V}$-module
\begin{equation}\label{eqn:TwistedModules}
	\gamma^* \brac{\Mod{M}} = \setbar[\big]{\gamma^* \brac{m}}{m \in \Mod{M}}, \quad \text{with $\VOA{V}$-action}
	\quad A(z) \cdot \gamma^* \brac{m} = \gamma^* \brac[\Big]{\gamma^{-1} \brac[\big]{A(z)} m}.
\end{equation}
This defines an invertible functor from $\VOA{V}\cMod$ to $\VOA{V}\cMod$ that we shall also denote by $\gamma$.
Li proved in \cite{Li97} that the spectral flow map \eqref{eqn:LiSpecFlow} also defines a twisting functor in this way.
In this form, we can consider spectral flow functors on appropriate categories of $\bgvoa$- and $\lvoa$-modules, see \cite{RidBos14,AdaRea17,AKR21}.
We will only need the latter and will describe it in detail in \zcref{sec:specflows} once we have fixed the parameters for inverse reduction.

\section{Reduction without spectral flow}\label{sec:CompositionTypeARankOne}

It is interesting to consider the composition of the reduction and inverse-reduction functors for $\usltwo$.
We refer to any composition of a reduction and an inverse-reduction functor as a \emph{composition functor}.
As the inverse reduction of a $\uvir$-module $\Mod{M}$ results in a spectral flow of a \fr\ $\usltwo$-module \cite{AdaRea17}, we can use these compositions to study the reduction of these modules.
In this \lcnamezcref{sec:CompositionTypeARankOne}, we first analyse the \emph{untwisted} case in which the $\usltwo$-module has no spectral flow (meaning that its conformal weights are bounded below) before turning to the case with non-trivial spectral flow (which turns out to be easier).

\subsection{Virasoro and $\usltwo$ vertex algebras}\label{sec:ReviewRankOne}

The Virasoro algebra is the complex Lie algebra
\begin{equation}
	\vir = \cspn \setbar{ \Tvir_n, C}{n \in \ZZ }, \quad \text{such that} \quad
	\comm{\Tvir_m}{\Tvir_n} = \brac{m-n} \Tvir_{m+n} + \frac{1}{12} \brac{m^3-m} \kdelta{m+n}{0} C
\end{equation}
and $C$ is central.
Consider the Verma module of the Virasoro algebra for which the $\Tvir_0$-eigenvalue of the \hwv\ $\ket{0}$ is $0$ and $C$ acts as multiplication by $c^{\vir} \in \CC$ (the central charge).
Its quotient by the submodule generated by $\Tvir_{-1} \ket{0}$ may be equipped with the structure of a \voa, known as the universal Virasoro \voa.
It is is freely generated by the \emt\ $\Tvir(z) = \sum_{n \in \ZZ} \Tvir_n z^{-n-2}$.


Consider next the complex Lie algebra $\sltwo$ spanned by the generators $e$, $h$ and $f$ satisfying
\begin{equation}
  \comm{h}{e}=2e, \quad \comm{e}{f}=h, \quad \comm{h}{f}=-2f.
\end{equation}
The rescaled Killing form is zero for all combinations of basis elements besides
\begin{equation}
  \killing{h}{h}=2 \quad \text{and} \quad \killing{e}{f}=\killing{f}{e}=1.
\end{equation}
This data defines the universal affine \voa\ $\usltwo$, $\kk\ne-2$, as per \zcref{subsec:QHR}.

There is a single non-zero nilpotent orbit in $\sltwo$.
Taking its representative to be $f$, the BRST complex $\brac*{ C^{\bullet}, \diff }$ is 
\begin{equation}\label{eqn:sltwoBRSTComplex}
  C^{\bullet} = \usltwo \otimes \bcvoa, \quad
  \cBRST(z) = \brac[\big]{ e(z) + \wun } \otimes \varphi^*(z),
\end{equation}
where $\diff$ is the zero mode of $\cBRST(z)$ as usual.
The corresponding \walg\ $\uwalg{\kk}{\sltwo, f}$ is the universal Virasoro \voa\ of central charge
\begin{equation}
	c^{\vir}_{\kk} = -\frac{(2\kk+1)(3\kk+4)}{\kk+2}
	= 13 - 6\,\brac*{\kk+2 + \frac{1}{\kk+2}}, \quad \text{for}\ \kk\ne-2.
\end{equation}
We shall denote this \voa\ by $\uvir$, for brevity.

Relaxed \hw\ $\usltwo$-modules have a long history in the theory of \voas\ and conformal field theory, see for example \cite{GawNon91,AM95,FST98,TesStr97,MalStr00,GabFus01,KR21}.
We define these modules, following \cite{RW15,KR19}, using the $\ZZ$-grading on $\asltwo$ induced by the adjoint action of the operator $T^{\sltwo}_0$.
Denoting the homogeneous component of eigenvalue $-n$ by $\brac{\asltwo}_n$, there is a generalised triangular decomposition
\begin{equation}
	\asltwo = \brac{\asltwo}_- \oplus \brac{\asltwo}_0 \oplus \brac{\asltwo}_+,
	\quad \text{where} \ \brac{\asltwo}_{\pm} = \bigoplus_{\pm n>0} \brac{\asltwo}_n.
\end{equation}
A \emph{\rhwv} in a $\usltwo$-module is an eigenvector of $h_0$ and generalised eigenvector of $T^{\sltwo}_0 = T^{\sltwo,0}_0$ that is, in addition, annihilated by $\brac{\asltwo}_+$.
A \emph{\rhwm} is a module generated by a single \rhwv.
If the \rhwm\ has no generating \rhwv\ annihilated by either $e_0$ or $f_0$, then we will refer to it as being \emph{\fr}.

Inverse reduction was first considered for $\usltwo$ by Semikhatov in \cite{SemMFF93, Sem94}.
There, he introduced a non-zero homomorphism
\begin{equation} \label{eqn:semikhatov}
	\Phi\colon \usltwo \to \uvir \otimes \lvoa
\end{equation}
of \vas.
(Referring back to \eqref{eqn:iqhrembedding}, we have $M=0$ and $N=1$.)
This is in fact injective for all levels (not just generically), see \cite{AKR21}.
Explicitly, the embedding $\Phi$ may be chosen to have the form
\begin{equation}\label{eqn:sltwoIQHRembedding}
	e(z) \mapsto e^c(z), \quad
	h(z) \mapsto 2 b(z), \quad
	f(z) \mapsto \no{ \brac[\big]{ (\kk+2) \Tvir(z) - (\kk+1) \partial a(z) - a(z) a(z) } e^{-c}(z) },
\end{equation}
where we introduce a convenient orthogonal set in $\pi$ given by
\begin{equation}\label{eqn:AltHeisenbergChoices}
	a = -\frac{\kk}{4} c + \frac{1}{2} d \quad \text{and} \quad b = +\frac{\kk}{4} c + \frac{1}{2} d.
\end{equation}
(It is clearly an orthogonal basis if $\kk\ne0$.)
This embedding is also conformal if we take the Sugawara \emt\ $T^{\sltwo}(z)$ for $\usltwo$ and
\begin{equation}\label{eqn:sltwoIQHRemt}
	T^{\lvoa}(z) = T^{\lvoa,\kk/4,-1/2}(z) = \frac{1}{2} \no{c(z) d(z)} - \pd a(z)
	= \frac{1}{2} \no{c(z) d(z)} + \frac{\kk}{4} \partial c(z) - \frac{1}{2} \partial d(z)
\end{equation}
for the half-lattice \va\ $\lvoa$.

As per the discussion in \zcref{subsec:IQHR}, this choice results in the central charge of $\lvoa$ being $c^{\lvoa} = 2+6\kk$.
The conformal weight of $e^{nc}$, $n \in \ZZ$, is $n$.
It is easy to check from \eqref{eq:actonexp} that the simple $\lvoa$-module
\begin{equation} \label{eq:defPilambda}
	\wlatt = \lvoa \cdot e^{(\mu+\kk/4) c - d/2} = \lvoa \cdot e^{-a + \mu c}, \quad \mu \in [\lambda] \in \CC/\ZZ,
\end{equation}
has conformal weights that are bounded below (with minimum $\frac{\kk}{4}$).
A convenient basis for $\wlatt$ is given by \cite{AKR21}
\begin{equation}\label{eqn:WeightPiBasis}
	\setbar[\big]{d_{-\alpha} e^c_{-\beta} e^{-a+\mu c}}{\alpha, \beta \in \partitions, \ \mu \in [\lambda]}.
\end{equation}
Note that the generating state $\ee^{-a + \mu c}$ has $h_0$-eigenvalue $2\mu$ under the identifications of \eqref{eqn:sltwoIQHRembedding}.

In this \lcnamezcref{sec:ReviewRankOne}, we consider the inverse-reduction functors $\IH = \IH[\wlatt]$, for $[\lambda] \in \CC/\ZZ$.
As the conformal weights of these modules are bounded below, it is easy to see that for any simple $\uvir$-module $\Mod{M}$, the inverse reduction $\IH(\Mod{M})$ is a \fr\ $\usltwo$-module \cite{AdaRea17}.
In fact, more is true: $\IH(\Mod{M})$ is \emph{almost simple}, see \cite{AKR21}.
Practically, this means that inverse reduction constructs $\usltwo$-modules that are simple for all but finitely many $[\lambda] \in \CC/\ZZ$.

\subsection{Reduction for relaxed $\usltwo$-modules} \label{subsec:QHR_relaxed}

We now consider how the reduction functor $\QH$ acts on the image of the inverse-reduction functors $\IH$, where $[\lambda] \in \CC/\ZZ$.
More precisely, we consider the category of $\uvir$-modules consisting of weight modules such that $\Tvir_0$ acts with finite-length Jordan chains.
The image under the inverse-reduction functors $\IH$ of this category includes the simple \fr\ $\usltwo$-modules, as every such module is isomorphic to some $\IH(\Mod{M})$ with $[\lambda] \in \CC/\ZZ$ and some simple $\Mod{M} \in \uvir\cMod$ \cite{AdaRea17,AKR23}.
This image also includes certain non-semisimple $\usltwo$-modules, which we will discuss in more detail later (see \zcref{sec:LogMods_review}).

Combining \zcref{eqn:sltwoBRSTComplex,eqn:semikhatov}, the BRST complex $C$ and differential $\diff = \oint_0 Q(z)\,\dd z$ are given by
\begin{equation}\label{eqn:BRSTComplexDiff}
	C^{\bullet} = {\Mod{M}} \otimes \wlatt \otimes \bcvoa, \quad
	\cBRST(z) = \brac[\big]{ e^c(z) + \wun } \otimes \varphi^*(z).
\end{equation}
Here, we regard the complex $C$ as a $\VOA{V}$-module, where $\VOA{V} = \uvir \otimes \lvoa \otimes \bcvoa$.

Before computing the BRST cohomology, we choose the conformal structure of $\VOA{V}$ so that $\cBRST(z)$ is homogeneous and has conformal weight $1$.
This results in the \emt
\begin{equation}\label{eqn:QHRconfstructure}
	\widetilde{T}(z)
	= T^{\sltwo,h/2}(z) + T^{\bcvoa,0}(z)
	= \Tvir(z) + \frac{1}{2} \no{c(z) d(z)} + \frac{\kk}{2} \pd c(z) + \no{\pd \varphi(z) \varphi^*(z)}.
\end{equation}
With respect to $\widetilde{T}$, the corresponding conformal weights of the strong generators of $\VOA{V}$ are as follows:
\begin{equation} \label{eqn:confwts}
	\begin{tabular}{C|CCCCCC}
		 & \Tvir & e^{nc}\ \text{($n\in\ZZ$)} & c & d & \varphi & \varphi^* \\
		\hline
		\Delta & 2 & 0 & 1 & 1 & 0 & 1
	\end{tabular}
	\:.
\end{equation}

Defining $\wtd(z) = \frac{1}{2} d(z) + \no{\varphi(z) \varphi^*(z)}$ for later convenience, the BRST complex $C$ has a basis given by
\begin{equation}\label{eqn:NoSpecBRSTBasis}
	\setbar[\big]{ \wtd_{-\alpha} e^c_{-\beta} \varphi_{-\kappa+1} \varphi^*_{-\nu} \genvec{m}{\mu} }{ \alpha, \beta \in {\partitions}, \ \kappa, \nu \in {\partitionsFerm}, \ \mu \in [\lambda] , \ m \in \mbasis },
\end{equation}
where $\genvec{m}{\mu} = m \otimes \ket{\mu}$, $\ket{\mu} = e^{-a+\mu c} \otimes \bcvac$, $\bcvac$ is the vacuum of $\bcvoa$, and $\mbasis$ is some basis of $\Mod{M}$.
(Here, we expand our fields into modes using the conformal weight with respect to $\widetilde{T}$ and employ the partition conventions of \zcref{subsec:partitions}.) Note that the differential $\diff$ annihilates states of the form $\genvec{m}{\mu}$ and its only non-zero commutation relations with the modes of the strong generators are
\begin{equation}\label{eqn:sltwoDiffAction}
	\comm{\diff}{\varphi_n} = e^{c}_n + \kdelta{n}{0} \wun, \quad \comm{\diff}{\wtd_n} = \varphi^*_n.
\end{equation}

We now briefly outline the strategy for computing the cohomology.
\begin{enumerate}
  \item\label{it:untidy} \textbf{Determine the action of the differential from the module structure.}
  As explained in \zcref{rem:NoQuasiIsoComplexes}, we only need the module structure of the BRST complex to determine the action of the differential $\diff$ on states $v \in C$.
  After computing $\diff v$, for arbitrary $v$, we may forget this module structure and regard $\brac{C, \diff}$ as a differential complex of graded vector spaces.
  This allows for factorisations of $C$ that would not be available if we wished to preserve module structures.
  Our first task is then the most untidy: we specify completely the action of $\diff$ on $C$.
  \item\label{it:gauge} \textbf{Factorise the complex and simplify the gauged lattice subcomplex.}
  We next note that $C$ factorises into a tensor product $C_0 \otimes C_1$ of two differential subcomplexes of graded vector spaces.
  One factor, $C_1$ say, is isomorphic to the ``gauged lattice complex'' of \zcref{sec:CohomAppendix} and so has cohomology $\CC$ (concentrated in ghost grade $0$).
  \item\label{it:Li} \textbf{Pass to Li's spectral sequence.}
  To ameliorate the untidiness of the $\diff$-action, we pass to the filtered spectral sequence $\brac[\big]{E^{\Li}_R, \diff^{\Li}_R}_{R\ge0}$ associated with Li's filtration (see \zcref{def:LiFiltrationVOAMod} below).
  The action of the zeroth differential $\diff^{\Li}_0$ on $v \in C_0$ is significantly simpler than that of $\diff$.
  (We will see that this spectral sequence collapses at the first page, so the higher differentials are not needed to compute the cohomology.)
  \item\label{it:Cartan} \textbf{Factorise again and simplify the Cartan subcomplex.}
  The result of the previous cohomology calculation again factorises into two differential subcomplexes of graded vector spaces.
  This time, one factor is isomorphic to the ``Cartan complex'' of \zcref{sec:CohomAppendix} which also has cohomology $\CC$ (concentrated in ghost grade $0$).
  \item\label{it:finish} \textbf{Equip the cohomology with the structure of a $\uvir$-module.}
  This determines the cohomology of $E^{\Li}_0$ as a graded vector space.
  We now observe that the filtered spectral sequence of Li collapses, hence this also gives the cohomology of $C_0$, hence of $C$, again as a graded vector space.
  Applying \eqref{eqn:cohomModStructure} then gives the cohomology the structure of a $\uvir$-module, completing the calculation.
\end{enumerate}

Let us begin, as outlined above in step~\ref{it:untidy}, by computing the action of the BRST differential.
Consider therefore a general basis state of the form
\begin{equation}\label{eqn:generalPBWBasisState}
  v = \wtd_{-\alpha} e^c_{-\beta} \varphi_{-\kappa+1} \varphi^*_{-\nu} \genvec{m}{\mu}
  \quad \text{for some}\ \alpha, \beta \in \partitions,\ \kappa, \nu \in \partitionsFerm,\ \mu \in [\lambda]\ \text{and}\ m \in \mbasis.
\end{equation}
\begin{lemma}\label{lem:BRSTActionGenPBW}
  The BRST differential acts on the basis state $v$ as
  \begin{subequations} \label{eqn:diffActionExpand1}
    \begin{equation}
	    \diff v = v_1 + v_2 + v_3,
	  \end{equation}
	  where
	  \begin{equation}\label{eqn:diffActionExpand1Terms}
	    \begin{aligned}
	      v_1 &= \:\sum_{\mathclap{1 \le i \le \lenp{\alpha}}}\:
	        \brac[\big]{\wtd_{-\alpha_1} \cdots \omitmode{\wtd_{-\alpha_i}} \varphi^*_{-\alpha_i} \cdots \wtd_{-\alpha_{\lenp{\alpha}}}}
	        e^c_{-\beta} \varphi_{-\kappa+1} \varphi^*_{-\nu} \genvec{m}{\mu},
	      \\
	      v_2 &= \:\sum_{\mathclap{\substack{1 \le i \le \lenp{\kappa} \\ \kappa_i \ne 1}}}\:
        (-1)^{i-1} \wtd_{-\alpha} e^c_{-\beta}
	      \brac[\big]{\varphi_{-\kappa_1+1} \cdots \omitmode{\varphi_{-\kappa_i+1}} e^c_{-\kappa_i+1} \cdots \varphi_{-\kappa_{\lenp{\kappa}}+1}}
	      \varphi^*_{-\nu} \genvec{m}{\mu},
	      \\
	      v_3 &= (-1)^{\lenp{\kappa}-1} \wtd_{-\alpha} e^c_{-\beta}
	      \brac[\big]{\varphi_{-\kappa_1+1} \cdots \omitmode{\varphi_{-\kappa_{\lenp{\kappa}}+1}}}
	      \varphi^*_{-\nu} \brac[\big]{\genvec{m}{\mu} + \genvec{m}{\mu+1}} \kdelta{\kappa_{\lenp{\kappa}}}{1}
	    \end{aligned}
	  \end{equation}
	  and $\omitmode{\hphantom{\cdots}}$ indicates omission.
	\end{subequations}
\end{lemma}
\begin{proof}
  This follows by explicitly (anti)commuting $\diff$ past the modes in $v$, using \eqref{eqn:sltwoDiffAction}, and recalling that $\diff$ annihilates $\genvec{m}{\mu}$.
  Each $\wtd_{-n}$ thus contributes a term in which this mode is replaced by $\varphi^*_{-n}$, resulting in $v_1$.
  Similarly, for $n\ne0$, each $\varphi_{-n}$ gives a term in which this mode is replaced by $e^c_{-n}$, resulting in $v_2$.
  However, if $n=0$, then we get an additional term in which $\varphi_0$ is replaced by $\wun$.
  Moreover, the replacement $e^c_0$ then commutes past all the ghost modes, where it acts on $\genvec{m}{\mu}$ to give $\genvec{m}{\mu+1}$, resulting in $v_3$.
  Of course, there can be at most one $\varphi_0$-mode, since $\kappa\in\partitionsFerm$, so it exists if and only if $\kappa_{\lenp{\kappa}}=1$.
\end{proof}

The terms in \eqref{eqn:diffActionExpand1Terms} are not all linear combinations of the basis elements \eqref{eqn:generalPBWBasisState}.
However, we may (anti)commute the modes until the required order is restored.
It is clear that $v_3$ is already ordered, while $v_2$ becomes ordered by commuting, in each term, an $e^c$-mode through some $\varphi$- and $e^c$-modes, none of which result in additional terms.
For $v_1$ however, each term has a $\varphi^*$-mode that must be commuted through some $\wtd$-modes and then anticommuted past all the $\varphi$-modes and some of the other $\varphi^*$-modes.
(Of course, if the $\varphi^*$-mode being commuted happens to be equal to one of the other $\varphi^*$-modes, then the term is $0$.)
Since
\begin{equation} \label{eq:dphicomm}
	\comm{\varphi^*_m}{\wtd_n} = \varphi^*_{m+n},
\end{equation}
ordering $v_1$ correctly typically results in additional terms with fewer $\wtd$-modes.

We summarise this as follows.
\begin{lemma}\label{lem:BRSTActionGenPBWLi}
  The action of the differential on an arbitrary basis vector $v$ is given by
  \begin{subequations} \label{eqn:diffActionExpand2}
		\begin{equation}
	    \diff v = v_1^{(0)} + v_1^{(>0)} + v_2^{(0)} + v_3^{(0)},
	  \end{equation}
	  where the \rhs\ is ordered correctly, $v_3^{(0)} = v_3$, $v_2^{(0)} = v_2$,
	  \begin{equation}\label{eqn:diffActionExpand1OtherTerm}
	    v_1^{(0)} = \:\sum_{\mathclap{\substack{1 \le i \le \lenp{\alpha} \\ \alpha_i \notin \nu}}}\: \eps_i \,
	    \wtd_{-\alpha'(i)} e^c_{-\beta} \varphi_{-\kappa+1} \varphi^*_{-\nu'(i)} \genvec{m}{\mu}, \quad
	    \text{for some}\ \eps_i \in \set{\pm1},
	  \end{equation}
	  and $v_1^{(>0)}$ is a linear combination of monomials with fewer $\wtd$-modes than those of $v_1^{(0)}$ but the same $e^c$- and $\varphi$-modes.
	\end{subequations}
  Here, $\alpha'(i)$ is $\alpha$ with the part $\alpha_i$ removed and $\nu'(i)$ is $\nu$ with the part $\alpha_i$ added.
\end{lemma}

At this point, we forget the $\VOA{V}$-module structure on $C$ and treat $\brac{C, \diff}$ as a differential complex of graded vector spaces.
We are now at step~\ref{it:gauge} of our calculation and so introduce
\begin{equation}\label{eqn:NoSpecFlowDecomp1}
	\begin{aligned}
		C_0 &= \cspn \setbar[\Big]{ \wtd_{-\alpha} \varphi^*_{-\nu} \dket{0} }{ \alpha \in \partitions, \ \nu \in \partitionsFerm }
		\\ \text{and} \quad
		C_1 &= \cspn \setbar[\Big]{ e^c_{-\beta} \varphi_{-\kappa+1} \dket{\mu} }{ \beta \in \partitions, \ \kappa \in \partitionsFerm, \ \mu \in [\lambda] }.
	\end{aligned}
\end{equation}
Here, the basis vectors $\wtd_{-\alpha} \varphi^*_{-\nu} \dket{0}$ and $e^c_{-\beta} \varphi_{-\kappa+1} \dket{\mu}$ are assigned ghost grades $\lenp{\nu}$ and $-\lenp{\kappa}$, respectively.
There is an obvious vector-space isomorphism $\Psi \colon C \to \Mod{M} \otimes C_0 \otimes C_1$ given by
\begin{equation}\label{eqn:TensorDecompOneIsomorphNoSpec}
	\Psi(v) = \Psi_{\Mod{M}}(v) \otimes \Psi_0(v) \otimes \Psi_1(v), \quad \text{where} \quad \left\{
	\begin{aligned}
		&\Psi_{\Mod{M}}\brac[\big]{\wtd_{-\alpha} e^c_{-\beta} \varphi_{-\kappa+1} \varphi^*_{-\nu} \genvec{m}{\mu}} = m, \\
		&\Psi_0\brac[\big]{\wtd_{-\alpha} e^c_{-\beta} \varphi_{-\kappa+1} \varphi^*_{-\nu} \genvec{m}{\mu}} = \wtd_{-\alpha} \varphi^*_{-\nu} \dket{0}, \\
		&\Psi_1\brac[\big]{\wtd_{-\alpha} e^c_{-\beta} \varphi_{-\kappa+1} \varphi^*_{-\nu} \genvec{m}{\mu}} = e^c_{-\beta} \varphi_{-\kappa+1} \dket{\mu}.
	\end{aligned}
	\right.
\end{equation}
This preserves the ghost grade if $\Mod{M}$ is assigned zero grade.

Of course, $\Mod{M}$ may be given the structure of a differential complex by equipping it with the zero differential.
We shall moreover define $\diff_i$ on $C_i$, for $i=0,1$, by
\begin{equation}\label{eqn:DiffDecompFactors}
	\diff_i \Psi_i(v) = \Psi_i(\diff v).
\end{equation}
It is easy to check that this makes $C_0$ and $C_1$ into differential complexes.
\begin{lemma}\label{lem:NoSpecFirstDecomp}
	We have $C \cong \Mod{M} \otimes C_0 \otimes C_1$, as differential complexes, and $\cohom{n}{C, \diff} \cong \Mod{M} \otimes \cohom{n}{C_0, \diff_0}$, for all $n \in \ZZ$.
\end{lemma}
\begin{proof}
	Observe from \eqref{eqn:diffActionExpand1} and \eqref{eqn:diffActionExpand2} that when acting with $\diff$ on the basis vector $v$ of \eqref{eqn:generalPBWBasisState}, every term of $v_1^{(0)}$ and $v_1^{(>0)}$ has the same $e^c$- and $\varphi$-modes as $v$, while every term of $v_2^{(0)}$ and $v_3^{(0)}$ has the same $\wtd$- and $\varphi^*$-modes as $v$.
	Because of this, $\Psi$ is an isomorphism of differential complexes:
	\begin{equation}
		\begin{split}
			\Psi(\diff v)
			&= \Psi\brac[\big]{v_1^{(0)} + v_1^{(>0)}} + \Psi\brac[\big]{v_2^{(0)} + v_3^{(0)}} \\
			&= m \otimes \Psi_0\brac[\big]{v_1^{(0)} + v_1^{(>0)}} \otimes \Psi_1(v) + m \otimes \Psi_0(v) \otimes \Psi_1\brac[\big]{v_2^{(0)} + v_3^{(0)}} \\
			&= m \otimes \diff_0\Psi_0(v) \otimes \Psi_1(v) + m \otimes \Psi_0(v) \otimes \diff_1\Psi_1(v) \\
			&= (0 \otimes \wun \otimes \wun + \wun \otimes \diff_0 \otimes \wun + \wun \otimes \wun \otimes \diff_1) \Psi(v).
		\end{split}
	\end{equation}
	This proves the first claim.
	For the second, it is clear that $C_1$ is isomorphic to the gauged lattice complex of \zcref{sec:CohomAppendix}, so \zcref{prop:GaugedComplexLatt} yields $\cohom{n}{C_1, \diff_1} \cong \CC \kdelta{n}{0}$.
	Applying \tkf\ then gives the desired result because $\diff$ annihilates $\Mod{M}$.
\end{proof}

\begin{remark} \label{rem:WhyIntroduceDoubleKets}
	As we have seen, it is tempting to formally identify the basis vectors used above to define $C_0$ and $C_1$ with parts of the basis vectors \eqref{eqn:generalPBWBasisState} of $C$.
	However, we do not have to do this and, given our vector-space philosophy from step~\ref{it:untidy}, perhaps we should not.
	It turns out to make little difference for the rest of this calculation.
	However, this freedom turns out to be useful for our computations when we allow spectral flow (and indeed when generalising to higher-rank \walgs).
	For this reason, we have explicitly disambiguated our vectors by using the notation $\dket{0}$ and $\dket{\mu}$, instead of $\ket{0}$ or $\ket{\mu}$.
\end{remark}

Next, we use Li's filtration (step~\ref{it:Li}) to help compute the cohomology of $C_0$.
This was introduced in \cite{Li05} as a canonical decreasing filtration of a module over an arbitrary \va.
\begin{definition}[\cite{Li05}]\label{def:LiFiltrationVOAMod}
	Let $\set{v^i}_{i \in I}$ be a set of strong generators for a \va\ $\VOA{V}$ and let $\Mod{N}$ be a $\VOA{V}$-module.
	Then, Li's filtration $\Lif\,\Mod{N} = \brac[\big]{\Mod{N} = \Lif^0 \Mod{N} \supseteq \Lif^1 \Mod{N} \supseteq \cdots}$ is defined by
	\begin{equation}
		\Lif^p \Mod{N} = \cspn \setbar[\Big]{ v^{i_1}_{\brac{-n_1-1}} \dots v^{i_r}_{\brac{-n_r-1}} w }{r \ge 0, \ i_1, \dots, i_r \in I, \ n_1, \dots, n_r \ge 0, \ \sum_{j=1}^r n_j \ge p, \ w \in \Mod{N}},
	\end{equation}
	for $p \in \NN$.
	We denote the associated graded space of $\Lif\, \Mod{N}$ by $\Ligr \Mod{N}$ so that $\Ligr^p \Mod{N} = \Lif^p \Mod{N} \big/ \Lif^{p+1} \Mod{N}$.
\end{definition}

From \eqref{eqn:confwts}, it follows that the modes of $C_0$ are graded by Li's filtration as follows:
\begin{equation}\label{eqn:LiFiltBRSTComplex}
	\Ligr \wtd_{-n} = n-1, \quad
	\Ligr \varphi^*_{-n} = n-1.
\end{equation}
(Of course, we assign zero grade to the state $\dket{0}$.)
From this, it is easy to check that the non-zero (anti)commutators \eqref{eqn:sltwoDiffAction} with $\diff$, leading to \eqref{eqn:diffActionExpand1}, preserve the Li grade, while the commutator \eqref{eq:dphicomm} \emph{increases} the Li grade by $1$.
It follows that $\diff$ (and hence $\diff_0$) is compatible with Li's filtration of $C$ (and hence $C_0$), in the sense of \zcref{def:compatibleFilt}.

To compute the cohomology of $C_0$, we turn to the filtered spectral sequence defined by Li's filtration.
The zeroth page is just the associated graded space with the usual bigrading:
\begin{equation}
	(E_0^{\Li})^{p,q} = \Ligr^p C_0^{p+q}, \quad p \in \NN,\ q \in \ZZ.
\end{equation}
Additionally, as $E^{\Li}_0$ does not involve $\varphi$-modes, the ghost grade is non-negative: $p+q\ge0$.
This is illustrated in \zcref{fig:SpectralFlowSpectralSeqPages}, from which it is clear that the spectral sequence is not bounded.
The convergence of unbounded spectral sequences is reviewed in \zcref{sec:UnboundedSpecSequences}.

The action of the zeroth-page differential $\diff^{\Li}_0$ on $\Psi_0(v) \in C_0$ is extracted from that of $\diff$ on $v$ by dropping $v_2^{(0)}$ and $v_3^{(0)}$ while ignoring any ordered terms of strictly greater Li grade.
As above, commuting with $\diff$ preserves the Li grade, while commuting $\varphi^*_m$ and $\wtd_n$ increases it.
The upshot is that
\begin{equation} \label{eqn:LiDiffTermsNoSpecFlow}
	\diff^{\Li}_0 \Psi_0(v) = \Psi_0\brac[\big]{v_1^{(0)}},
\end{equation}
referring back to the decompositions of \zcref{lem:BRSTActionGenPBW,lem:BRSTActionGenPBWLi}.
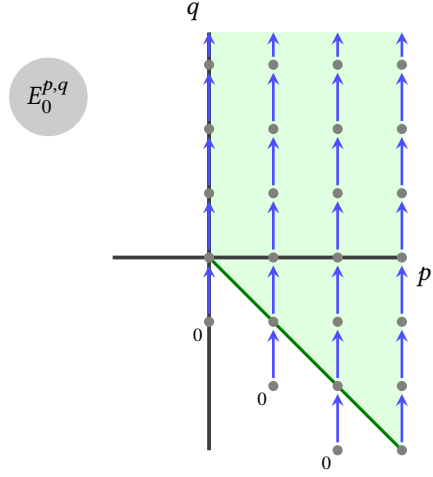
\begin{figure}
\centering
	\begin{tikzpicture}[scale=0.85]
		\draw[color=green!30, fill=green!30, opacity=0.4] (3, -3) -- (0, 0) -- (0, 3.5) -- (3, 3.5) -- cycle;
		\draw[color=green!50!black, very thick] (3, -3) -- (0, 0) -- (0, 3.5);
		\draw[line width=0.5mm, gray!10!black!80!] (-1.5, 0) -- (3, 0);
		\draw[line width=0.5mm, gray!10!black!80!] (0, 3.5) -- (0, -3);
		\node[anchor=south east] at (0, 3.55) {$q$};
		\node[anchor=north west] at (3.1, 0) {$p$};
		\node[circle, fill=gray!40!white] at (-2.5, 2.5) {$E_0^{p, q}$};
		\colorlet{lineColour}{gray!40!red}
		\newcommand{\tempScale}{0.7}

		\foreach [evaluate={\a=int(\x); \b=int(\x-1);}] \x in {0, ..., 2} {
			\foreach [evaluate={\c=int(-\y); \d=int(\y-1);}] \y in {1, ..., 5}{
				\ifthenelse{\a > \d}{
					\filldraw[dotcolour] (\a, \c) circle (2pt);
					}
				}
			}
		\filldraw[dotcolour] (3, -1) circle (2pt);
		\filldraw[dotcolour] (3, -2) circle (2pt);
		\filldraw[dotcolour] (3, -3) circle (2pt);
		\foreach [evaluate={\a=int(\x); \b=int(\x-1);}] \x in {0, ..., 3} {
			\foreach [evaluate={\c=int(-\y); \d=int(\y-1);}] \y in {-3, ..., -1}{
				\ifthenelse{\a > \d}{
					\filldraw[dotcolour] (\a, \c) circle (2pt);
					}
				}
			}
		\foreach [evaluate={\a=int(\x); \b=int(\x-1);}] \x in {0, ..., 3} {
			\filldraw[dotcolour] (\a, 0) circle (2pt);
			}
		\foreach [evaluate={\a=int(\x); \b=int(\x-1);}] \x in {2, ..., 3} {
			\foreach [evaluate={\c=int(-\y); \d=int(\y+1); \dd=int(-\y-1);}] \y in {-4, ..., 4}{
				\ifthenelse{\a = \d}{
					\node[anchor=north east, scale=\tempScale] at (\b, \dd) {$0$};
					\filldraw[dotcolour] (\b, \dd) circle (2pt);
					}
				}
			}
		\node[anchor=north east, scale=\tempScale] at (0, -1) {$0$};
		\filldraw[dotcolour] (0, -1) circle (2pt);

        \foreach [evaluate={\a=int(\x); \f=int(\x+2)}] \x in {0, ...,  3}{
            \foreach [evaluate={\b=int(-\y); \c=int(\y); \d=int(-\y+1); \e=int(\y-1);}] \y in {-2, ...,  3}{
                \ifthenelse{\c < \f}{
					\draw[draw=arrowcolour, -stealth, shorten <= 3pt, shorten >= 3pt, line width=1pt] (\a, \b) -- (\a, \d);
                    }
                }
            }
        \foreach [evaluate={\a=int(\x); \f=int(\x+2)}] \x in {0, ...,  3}{
			\draw[draw=arrowcolour, -stealth, shorten <= 3pt, shorten >= 12pt, line width=1pt] (\a, 3) -- (\a, 4);
            }
	\end{tikzpicture}
	\caption{%
		The zeroth page $E^{\Li}_0$ of Li's spectral sequence for $C_0$.
		Here the arrows indicate the action of the zeroth differential $\diff^{\Li}_0$.
		We have also shaded green the part of the $pq$-plane in which the bigraded subspaces $(E^{\Li}_0)^{p,q}$ are non-zero.
	}\label{fig:SpectralFlowSpectralSeqPages}
\end{figure}

In what follows, let $\HLif$ denote the filtration on the cohomology $H = \cohom{}{C_0, \diff}$, induced from Li's filtration.
Given that the spectral sequence is half-plane, applying \zcref{thm:WeakConvg_HalfPlane} shows that it converges weakly to the associated graded space $\Gr H$ of the cohomology $H$.
However, if we want to reconstruct $H$ from $\Gr H$, then it is necessary to show convergence, not just weak convergence.
This involves proving that the induced filtration $\HLif$ is Hausdorff.
Once this has been done, we must also account for the possibility that $H$ contains elements which lie in infinitely many filtered components $\HLif^p H$ ($\Gr H$ is insensitive to such elements).
In fact, both tasks can be addressed by relating Li's filtration to another filtration based on conformal weights.

As we only work with $\uvir$-modules $\Mod{M}$ such that $\Tvir_0$ acts with finite-length Jordan chains (\zcref{subsec:QHR_relaxed}), $\Mod{M}$ decomposes into a direct sum of generalised $\Tvir_0$-eigenspaces $\Mod{M}_{(\Delta)}$.
Explicitly,
\begin{equation}\label{eqn:GenConfEigen}
  \Mod{M} = \bigoplus_{\virwt \in \CC} \Mod{M}_{[\Delta]},
  \quad \text{where} \quad
  \Mod{M}_{[\Delta]} = \setbar[\big]{m \in \Mod{M}}{(T_0 - \Delta)^N m = 0 \ \text{for some} \ N \in \ZZ_{\ge 1}}.
\end{equation}
As the Semikhatov embedding \eqref{eqn:sltwoIQHRembedding} is in particular a conformal embedding, see \eqref{eqn:sltwoIQHRemt}, and the $\lvoa$-modules appearing here are semisimple, the resulting $\usltwo$-modules $\IH \brac{\Mod{M}}$ will also have finite-length Jordan chains and so admit a decomposition similar to that of \eqref{eqn:GenConfEigen}.
(This remains true even if we shift the \emt\ by derivatives of Heisenberg fields, because our modules are weight.)

The ghost vacuum module $\bcvoa$ is also semisimple, so each of the tensor factors comprising the BRST complex $C$ admits a decomposition into generalised eigenspaces with respect to the zero mode of their respective energy-momentum tensor.
The same is therefore true for $C$ with respect to $\widetilde{T}_0$.
Define $C_{0, [\Delta]} \subset C_0$ to be the span of the states $v = \widetilde{d}_{-\alpha} \varphi^*_{-\nu} \dket{0}$ satisfying $\abs{\alpha} + \abs{\nu} = \Delta$.

We say that the filtration $\Lif$ is \emph{conformally bounded} if, for each $\Delta \in \ZZ_{\ge 0}$, there exists $ p(\Delta) \in \ZZ$ such that
\begin{equation}
\Lif^p C_0 \cap C_{0, [\Delta]} = 0, \quad \text{for all}\ p>p(\Delta).
\end{equation}

\begin{lemma}\label{lem:LiFiltLowerBoundConfFilt}
  Suppose that $\VOA{V}$ is a vertex operator algebra with strong generators $\set{a^i}_{i\in I}$ and that $\Mod{N}$ is a $\VOA{V}$-module with minimal conformal weight $\Delta_0$.
  Furthermore, suppose that $\Mod{K} \subseteq \Mod{N}$ is a subspace generated by states in some set $S$ acted upon by a vertex subalgebra that is strongly generated by elements $\set{a^j}_{j\in J}$, for some $J \subset I$, whose conformal weights are strictly positive.
  Then, Li's filtration $\Lif$ on $\Mod{K}$ is conformally bounded.
\end{lemma}
\begin{proof}
  Denote by $\Delta_i$ the conformal weight of the generator $a^i$, where $i \in I$.
  Let $x \in \Lif^p \Mod{K}$ have conformal weight~$\Delta$.
  We may then write $x$ as a linear combination of terms taking the form
  \begin{equation}
    a^{j_1}_{(-n_1-1)} \cdots a^{j_r}_{(-n_r-1)} w, \quad \text{where} \ r \ge 0, \ j_1, \dots, j_r \in J, \ n_1, \dots, n_r \ge 0, \ \sum_{j=1}^r n_j \ge p, \ \text{and}\ w \in \Mod{N}_{\Delta_w}.
  \end{equation}
  Each such term has conformal weight $\Delta = \Delta_w + \sum_{j=1}^r \brac{\Delta_j + n_j}$.
  As $\Delta_j > 0$ and $\Mod{K}$ has minimal conformal weight $\Delta_0$, we have $\Delta > \Delta_w + p \ge \Delta_0 + p$.
  It follows that $\Lif^p \Mod{K} \cap \Mod{K}_{[\Delta]} = 0$ whenever $p > \Delta-\Delta_0$.
  Given that we already have $\Lif^p \Mod{K} = 0$ for $p<0$, this proves the statement.
\end{proof}

Now we want to prove that $\HLif$ is Hausdorff, as this shows that there are no elements in the cohomology which belong to all filtered components.
However we must also address the possibility that the cohomology contains elements which have infinitely many filtered components, even if they do not lie in all filtered components.
Pleasingly, it turns out that \zcref{lem:LiFiltLowerBoundConfFilt} rules out both possibilities simultaneously.
\begin{proposition}\label{lem:LiFiltLowerBoundConfFilt_Cohom}
  If Li's filtration $\Lif$ on $C_0$ is conformally bounded, then so is the induced filtration $\HLif$ on $H$.
  In particular, $\HLif$ is Hausdorff and there is a vector space isomorphism $H \cong \HGr H$.
\end{proposition}
\begin{proof}
  Consider $x \in \HLif^p H$ so that, by \eqref{eqn:inducedFilt}, there exists some $\diff$-closed $v \in F^p C_0$ such that $x=\eqclass{v}$.
  Decompose $v$ into components of fixed conformal weight:
  \begin{equation}
    v = \sum_{i=0}^N v_{[\Delta_0 + i]}, \quad \text{where} \quad v_{[\Delta_0+i]} \in C_{0, [\Delta_0+i]}.
  \end{equation}
  From \zcref{lem:LiFiltLowerBoundConfFilt}, $v_{[\Delta_0+i]}$ lies in finitely many filtered components $\Lif^p C_0$ and so $\eqclass{v_{[\Delta_0+i]}}$ is in finitely many of the $\HLif^p H$.
  As there are only finitely many terms $\eqclass{v_{[\Delta_0+i]}}$ in the decomposition of $\eqclass{v}$, we see that $\eqclass{v}$ lies in finitely many filtered components as well.
  So, all elements of $H$ lie in finitely many filtered components for $\HLif$, meaning that there is a vector space isomorphism $H \cong \Gr H$.
  Moreover, no elements lie in all filtered components, so $\HLif$ is Hausdorff.
\end{proof}

\begin{remark}
	We remark that if we had imposed Li's filtration on the full BRST complex $C$, rather than just on $C_0$, then we would need a different argument for the convergence of the spectral sequence.
	This is because applying \zcref{lem:LiFiltLowerBoundConfFilt, lem:LiFiltLowerBoundConfFilt_Cohom} here relies on the positivity of the conformal weights of the generators of the complex, which fails for $\varphi$ and $e^c$.
\end{remark}

With this subtle convergence result in hand, we turn to step~\ref{it:Cartan}.
\begin{lemma} \label{lem:NoSpecSecondDecomp}
	We have $\cohom{n}{E^{\Li}_0, \diff^{\Li}_0} \cong \CC \kdelta{n}{0}$.
\end{lemma}
\begin{proof}
	This follows because $(E^{\Li}_0, \diff^{\Li}_0)$ is isomorphic to the Cartan complex of \zcref{sec:CohomAppendix}, so its cohomology is given by \zcref{prop:CartanComplexCohom}.
\end{proof}

Our main result for this \lcnamezcref{sec:CompositionTypeARankOne} is now obtained by completing step~\ref{it:finish}.
\begin{theorem}\label{thm:sltwocomposition}
	For any $\uvir$-module $\Mod{M}$ and any $[\lambda] \in \CC/\ZZ$, the BRST cohomology of $\IH(\Mod{M})$ is concentrated in degree $0$.
	Moreover, we have
	\begin{equation}
		\brac{ \QH \circ \IH} \brac{\Mod{M}} \cong \Mod{M}.
	\end{equation}
\end{theorem}
\begin{proof}
	By \zcref{lem:NoSpecSecondDecomp}, the first page of the spectral sequence for Li's filtration is $(E^{\Li}_1)^{p,q} = \CC \kdelta{p}{0} \kdelta{q}{0}$.
	The spectral sequence thus collapses and, as per \zcref{lem:LiFiltLowerBoundConfFilt_Cohom}, therefore converges to the cohomology of $C_0$.
	By \zcref{lem:NoSpecFirstDecomp,lem:NoSpecSecondDecomp}, we thus get vector-space isomorphisms
	\begin{equation}
		\cohom{n}{C,\diff}
		\cong \Mod{M} \otimes \cohom{n}{C_0, \diff_0}
		\cong \Mod{M} \otimes \cohom{n}{E^{\Li}_0, \diff^{\Li}_0}
		\cong \Mod{M} \kdelta{n}{0}.
	\end{equation}

	To conclude, observe that when $\Mod{M} = \uvir$ (the vacuum module), hence $C (\Mod{M}) = \VOA{V}$, the vector-space isomorphism $\cohom{0}{\VOA{V}, \diff} \simeq \uvir$ lifts to an isomorphism of vertex algebras using the action given in \eqref{eqn:cohomModStructure}.
	More generally, this action results in $\cohom[\big]{n}{C(\Mod{M}), \diff} \simeq \Mod{M} \kdelta{n}{0}$ as $\uvir$-modules.
\end{proof}

As every simple \fr\ $\usltwo$-module is obtained from some $\uvir$-module by inverse \qhr, it follows from this \lcnamezcref{thm:sltwocomposition} that the \qhr\ of such a $\usltwo$-module is just the corresponding $\uvir$-module.
This remains true for the non-semisimple almost-simple \fr\ $\usltwo$-modules on which $e_0$ acts injectively, because each is in the image of some inverse-reduction functor $\IH$.

\section{Reduction with spectral flow}\label{sec:CompositionSpecFlow}

Having computed the \qhr\ of the \fr\ $\usltwo$-modules obtained by inverse reduction, we now turn to the computations when we first twist by a non-trivial spectral flow functor.
Surprisingly, it turns out that the computations are easier in this case.

\subsection{Spectral flows} \label{sec:specflows}

The spectral flow automorphisms $\specfsltwo{\ell}$ for the Lie algebra of modes of $\usltwo$ are indexed by $\ell \in \ZZ$ and are explicitly given by
\begin{equation}\label{eqn:specFlowsltwo}
	\specfsltwo{\ell} \brac{e_n} = e_{n-\ell},
	\quad
	\specfsltwo{\ell} \brac{f_n} = f_{n+\ell},
	\quad
	\specfsltwo{\ell} \brac{h_n} = h_n - \ell \kk \delta_{n=0} \wun,
	\quad
	\specfsltwo{\ell} \brac{T^{\sltwo}_n} = T^{\sltwo}_n - \frac{1}{2} \ell h_n + \frac{\kk}{4} \ell^2 \kdelta{n}{0} \wun.
\end{equation}
This corresponds to flowing with $j(z) = \frac{1}{2} \ell h(z)$ in \zcref{subsec:automorphisms}.

Li's theory of spectral flow also applies to the half-lattice \va\ $\lvoa$.
Its underlying Lie algebra of modes has a two-parameter family of spectral flow automorphisms corresponding to the rank-$2$ Heisenberg subalgebra in $\lvoa$.
For our purposes, it will be useful to consider the spectral flow automorphism associated to $j(z) = 
\ell b(z)$, where $b(z)$ was defined in \eqref{eqn:AltHeisenbergChoices}.
In this case, the automorphism is
\begin{equation}\label{eqn:specFlowLatt}
	\specflvoa{\ell} \brac{d_n} = d_n - \frac{\kk}{2} \ell \kdelta{n}{0} \wun,
	\quad
	\specflvoa{\ell} \brac{c_n} = c_n - \ell \kdelta{n}{0} \wun,
	\quad
	\specflvoa{\ell} \brac{e^{Nc}_n} = e^{Nc}_{n - \ell N},
	\quad
	\specflvoa{\ell} \brac{T^{\lvoa}_n} = T^{\lvoa}_n - \ell b_n + \frac{\kk}{4} \ell^2 \kdelta{n}{0} \wun.
\end{equation}

This choice is natural when studying spectral flows of $\usltwo$-modules using the inverse-reduction embedding $\Phi$ of \eqref{eqn:sltwoIQHRembedding}, as $h \in \usltwo$ maps to $2 b \in \uvir \otimes \lvoa$.
Indeed, it is easy to check that \cite{AdaRea17}
\begin{equation}
  \Phi \circ \specfsltwo{\ell} \cong \brac{ \id_{\vir} \otimes \specflvoa{\ell} } \circ \Phi.
\end{equation}

Given $[\lambda] \in \CC/\ZZ$ and $\ell \in \ZZ$, we will write $\specflatt$ instead of $\specflvoa{\ell} \brac{\wlatt}$ for convenience.
To obtain a basis for $\specflatt$, we apply the spectral flow automorphism \eqref{eqn:specFlowLatt} to the basis elements of \eqref{eqn:WeightPiBasis}:
\begin{equation}\label{eqn:SpecFlowWeightPiBasis}
	\setbar[\big]{ d_{-\alpha} e^c_{-\beta-\ell} \specflvoa{\ell}(\ee^{-a+\mu c}) }{ \alpha, \beta \in \partitions, \ \mu \in [\lambda] }.
\end{equation}
We note that when $n < -\ell$, the $e^c_n$ act freely on the generating states $\specflvoa{\ell}(\ee^{-a+\mu c})$.
Otherwise, they act as
\begin{equation}\label{eqn:SpecFlowActionOnBoundary}
	e^c_n  \specflvoa{\ell}(\ee^{-a+\mu c}) =
	\begin{cases*}
		0 & for $n>-\ell$, \\
		\specflvoa{\ell}(\ee^{-a+(\mu+1)c}) & for $n=-\ell$.
	\end{cases*}
\end{equation}

Spectrally flowed \fr\ $\usltwo$-modules may be obtained from the inverse-reduction functors $\specflatt$.
Explicitly, if $\Mod{M}$ is a non-zero $\uvir$-module, then
\begin{equation}
	\Mod{M} \otimes \specflatt \cong \specfsltwo{\ell} \brac{ \Mod{M} \otimes \wlatt }
\end{equation}
is the spectral flow of an almost-simple \fr\ $\usltwo$-module.
We therefore introduce a family
\begin{equation} \label{eqn:genIQHRFunctor}
  \IH^{\ell} \brac{\Mod{M}} = \Mod{M} \otimes \specflatt
\end{equation}
of inverse-reduction functors parameterised by the weight coset $[\lambda] \in \CC/\ZZ$ and spectral flow parameter $\ell \in \ZZ$.

Consider the BRST complex $C$ obtained by tensoring $\IH^{\ell}\brac{\Mod{M}}$ with the fermionic ghost vertex superalgebra.
With $\mbasis$ denoting a basis of $\Mod{M}$, a basis of this complex is given by
\begin{equation}\label{eqn:SpecFlowBRSTComplex}
	\setbar[\big]{ \wtd_{-\alpha} e^c_{-\beta-\ell} \varphi_{-\kappa+1} \varphi^*_{-\nu} \specfl{\ell} \brac[\big]{ \genvec{m}{\mu} } }{ \alpha, \beta \in \partitions, \ \kappa, \nu \in \partitionsFerm, \ m \in \mbasis, \ \mu \in [\lambda] },
\end{equation}
where we recall the $\wtd_n$ defined in \zcref{subsec:QHR_relaxed}.
Here, we denote by $\specfl{\ell}$ the spectral flow map $\id_{\Mod{M}} \otimes \specflvoa{\ell} \otimes \id_{\bcvoa}$.

From this point, it will be necessary to consider the cases of positive and negative spectral flow index $\ell$ separately.
For positive spectral flow, the generating states of the BRST complex are given by
\begin{equation}\label{eqn:defBRSTGenStatesSpecFlow1}
	\specfl{\ell} \brac[\big]{ \genvec{m}{\mu} } = m \otimes \specflvoa{\ell} \brac[\big]{ e^{-a + \mu c} } \otimes \bcvac, \quad \text{where } \ell > 0.
\end{equation}
We will not use this choice when $\ell<0$ (see \eqref{eqn:defBRSTGenStatesSpecFlow2} below).
To explain why, suppose (for the remainder of this subsection) that we do choose \eqref{eqn:defBRSTGenStatesSpecFlow1} for all $\ell \in \ZZ$.
The action of the differential would then be qualitatively different for positive and negative spectral flows:
\begin{equation}\label{eqn:SpecFlowDiffAction}
	\diff \specfl{\ell} \brac[\big]{ \genvec{m}{\mu} } =
	\begin{cases*}
		0, & for $\ell>0$, \\
		\displaystyle\sum_{1 \le n \le -\ell} \specfl{\ell} \brac[\big]{ e^c_{n+\ell} \varphi^*_{-n} \genvec{m}{\mu} }, & for $\ell<0$.
	\end{cases*}
\end{equation}
This non-zero action in the negative case is very inconvenient for cohomology calculations.
Indeed, it leads to a new term in the expression \eqref{eqn:diffActionExpand1} for $\diff v$ that involves both $e^c$- and $\varphi^*$-modes, which obstructs the factorisation in step~\ref{it:gauge} (from \zcref{subsec:QHR_relaxed}).

One could instead try to proceed by filtering the entire complex.
The most natural options are then to filter by the number of modes acting on the generating state or to use Li's filtration.
In both cases, the new term would have monomials whose filtration degrees differ by either $-1$ or $-2$, constrasting with the old terms whose degrees differ by non-negative integers.
It follows that the differential would not be compatible with either filtration.

\subsection{Reduction with positive spectral flow}\label{subsec:SpecFlowTypeARank1Pos}

Consider first the case of positive spectral flow parameter $\ell > 0$.
To compute the cohomology, we follow the strategy outlined in \zcref{subsec:QHR_relaxed}.
In fact, we will see that only the first two steps are required because the cohomology of one of the factors obtained in step~\ref{it:gauge} vanishes.

Step~\ref{it:untidy} proceeds as in \zcref{subsec:QHR_relaxed}, though the shifts in the mode indices lead to a few differences.
In particular, \eqref{eqn:sltwoDiffAction} shows that anticommuting $\diff$ past a $\varphi_n$ produces an $\ee^c_n$ that annihilates $\specfl{\ell} \brac[\big]{\genvec{m}{\mu}}$, if $n>-\ell$, and converts it into $\specfl{\ell} \brac[\big]{\genvec{m}{\mu+1}}$, if $n=-\ell$.
However, this anticommutation also produces an additional $\wun$ if $n=0$.
\zcref{lem:BRSTActionGenPBW} thus generalises as follows.
\begin{lemma} \label{lem:PosSpecFlowFirstDecomp}
  For $\ell>0$, the action of the BRST differential on a basis vector $v$ from \eqref{eqn:SpecFlowBRSTComplex} is
  \begin{equation}\label{eqn:diffActionExpand1PosSpec}
    \diff v = v_1 + v_2 + v_3 + v_4,
  \end{equation}
  where
  \begin{equation}\label{eqn:diffActionExpand1PosTermsPosSpec}
    \begin{aligned}
      v_1 &= \:\sum_{\mathclap{1 \le i \le \lenp{\alpha}}}\:
      \brac[\big]{ \wtd_{-\alpha_1} \cdots \omitmode{\wtd_{-\alpha_i}} \varphi^*_{-\alpha_i} \cdots \wtd_{-\alpha_{\lenp{\alpha}}} }
      e^c_{-\beta-\ell} \varphi_{-\kappa+1} \varphi^*_{-\nu} \specfl{\ell} \brac[\big]{ \genvec{m}{\mu} },
      \\
      v_2 &= \:\sum_{\mathclap{\substack{1 \le i \le \lenp{\kappa} \\ \kappa_i > \ell+1}}}\:
      (-1)^{i-1} \wtd_{-\alpha} e^c_{-\beta-\ell}
      \brac[\big]{ \varphi_{-\kappa_1+1} \cdots \omitmode{\varphi_{-\kappa_i+1}} e^c_{-\kappa_i+1} \cdots \varphi_{-\kappa_{\lenp{\kappa}}+1} }
      \varphi^*_{-\nu} \specfl{\ell} \brac[\big]{ \genvec{m}{\mu} },
      \\
      v_3 &= (-1)^{i-1} \wtd_{-\alpha} e^c_{-\beta-\ell}
      \brac[\big]{ \varphi_{-\kappa_1+1} \cdots \omitmode{\varphi_{-\kappa_i+1}} \cdots \varphi_{-\kappa_{\lenp{\kappa}}+1} }
      \varphi^*_{-\nu} \specfl{\ell} \brac[\big]{ \genvec{m}{\mu+1} } \kdelta{\kappa_i}{\ell+1},
      \\
      v_4 &= (-1)^{\lenp{\kappa}-1} \wtd_{-\alpha} e^c_{-\beta-\ell}
      \brac[\big]{ \varphi_{-\kappa_1+1} \cdots \omitmode{\varphi_{-\kappa_{\lenp{\kappa}}+1}} }
      \varphi^*_{-\nu} \specfl{\ell} \brac[\big]{\genvec{m}{\mu}} \kdelta{\kappa_{\lenp{\kappa}}}{1}.
    \end{aligned}
  \end{equation}
\end{lemma}
\noindent Here, we abuse notation by writing $\kdelta{\kappa_i}{\ell+1}$ to indicate that $v_3$ vanishes unless \emph{there exists} an $i$ such that the $i$-th part of $\kappa$ is $\ell+1$.

The terms of $v_1$ and $v_2$ in \eqref{eqn:diffActionExpand1PosTermsPosSpec} are not all linear combinations of the basis elements \eqref{eqn:defBRSTGenStatesSpecFlow1} (those of $v_3$ and $v_4$ are).
We can reorder them using (anti)commutation relations, as per \zcref{lem:BRSTActionGenPBWLi}, but this is not needed for what follows.
Instead, we forget the $\VOA{V}$-module structure on $C$ and regard $\brac{C, \diff}$ as a differential complex of graded vector spaces.
Step~\ref{it:gauge} now proceeds largely as in \zcref{subsec:QHR_relaxed}.
We introduce
\begin{equation}\label{eqn:PosSpecFlowInitialDecomp}
	\begin{aligned}
		C_0 &= \cspn \setbar[\big]{ \wtd_{-\alpha} \varphi^*_{-\nu} \dket{0} }{ \alpha \in \partitions, \ \nu \in \partitionsFerm }
		\\ \text{and} \quad
		C_1 &= \cspn \setbar[\big]{  e^c_{-\beta-\ell} \varphi_{-\kappa+1} \specfl{\ell} \brac[\big]{\dket{\mu}} }{ \beta \in \partitions, \ \kappa \in \partitionsFerm, \ \mu \in [\lambda] },
	\end{aligned}
\end{equation}
assigning the basis vectors $\wtd_{-\alpha} \varphi^*_{-\nu} \dket{0}$ and $e^c_{-\beta} \varphi_{-\kappa+1} \specfl{\ell} \brac[\big]{\dket{\mu}}$ the ghost grades $\lenp{\nu}$ and $-\lenp{\kappa}$, respectively.
$\Mod{M}$ is of course assigned ghost grade $0$.
We then have a vector-space isomorphism $C \to \Mod{M} \otimes C_0 \otimes C_1$ given by \eqref{eqn:TensorDecompOneIsomorphNoSpec} (with suitable adjustments to account for shifts in the modes for our new basis vectors).

We equip $\Mod{M}$ with the zero differential and define $\diff_i$ on $C_i$ by \eqref{eqn:DiffDecompFactors}, for $i=0,1$.
It follows that $\Mod{M}$, $C_0$ and $C_1$ are differential complexes.
Repeating the proof of \zcref{lem:NoSpecFirstDecomp}, we conclude that $C \cong \Mod{M} \otimes C_0 \otimes C_1$ as differential complexes and that
\begin{equation} \label{eqn:PosSpecFlowFirstDecomp}
  \cohom{n}{C, \diff} \cong \bigoplus_{p+q=n} \Mod{M} \otimes \cohom{p}{C_0, \diff_0} \otimes \cohom{q}{C_1, \diff_1}, \quad \text{for all } n \in \ZZ,
\end{equation}
by \tkf.

Unlike the analogous case in \zcref{subsec:QHR_relaxed}, $\brac{C_1, \diff_1}$ is not isomorphic to the gauged lattice complex.
In fact, it has trivial cohomology.
To show this, we note that the action of $\diff_1$ on $C_1$ may delete modes and change $\mu$, but it does not change the mode indices.
We may therefore factor $C_1$ into an infinite number of differential complexes, one for each possible mode index.
More explicitly, we have $C_1 \cong \bigotimes_{n \in \ZZ_{\le 0}} C_{1,n}$, where
\begin{equation}\label{eqn:PosSpecFlowModeDecomp}
	C_{1,n} =
	\begin{cases*}
		\cspn \setbar[\big]{ \varphi^s_n \dket{0} }{ s \in \set{0, 1} } & for $n > -\ell$,
		\\
		\cspn \setbar[\big]{ \varphi^s_n \specfl{\ell} \brac[\big]{\dket{\mu}} }{ s \in \set{0, 1}, \mu \in [\lambda] } & for $n=-\ell$,
		\\
		\cspn \setbar[\big]{ (e^c_n)^r \varphi^s_n \dket{0} }{ r \in \ZZ_{\ge 0}, \ s \in \set{0, 1} } & for $n<-\ell$
	\end{cases*}
\end{equation}
and the differential $\diff_{1,n}$ on $C_{1,n}$ is given by the appropriate restriction of $D_1$.
One may verify that $\cohom{p}{C_{1,n}, \diff_{1,n}} = 0$ for both $n=0$ and $n=-\ell$, however it will suffice to show only the first case.

\begin{lemma}\label{lem:PosSpecTrivialFactor}
  For all $p \in \ZZ$, we have $\cohom{p}{C_{1,0}, \diff_{1,0}} = 0$.
\end{lemma}
\begin{proof}
  The only non-zero terms of the complex $C_{1,0}$ are
  \begin{equation}
    C_{1,0}^{-1} = \CC \varphi_{-\ell} \dket{0}
    \quad \text{and} \quad
    C_{1,-\ell}^0 = \CC \dket{0}.
  \end{equation}
  Using \zcref{lem:PosSpecFlowFirstDecomp}, we compute that $\diff_{1,0} \varphi_{-\ell} \dket{0} = \dket{0}$.
  Consequently, $\ker \brac[\big]{\diff_{1,0} \colon C_{1,0}^{-1} \to C_{1,0}^0} = 0$ and $\im \brac[\big]{\diff_{1,0} \colon C_{1,0}^{-1} \to C_{1,0}^0} \cong C_{1,0}^0$.
  It follows that
  \begin{equation}\label{eqn:posSpecFlowZeroModeCohom}
    \cohom{p}{C_{1,0}, \diff_{1,0}} = 0, \quad \text{for both}\ p=0,-1. \qedhere
  \end{equation}
\end{proof}

\begin{theorem}\label{thm:PositiveSpecFlowCompositionsltwo}
	For all $\ell>0$, $\uvir$-modules $\Mod{M}$ and $[\lambda] \in \CC/\ZZ$, we have
	\begin{equation}
		\brac{\QH \circ \IH^{\ell}} \brac{\Mod{M}} = 0.
	\end{equation}
\end{theorem}
\begin{proof}
  By \zcref{lem:PosSpecTrivialFactor} and \tkf, we get $\cohom{p}{C_1, \diff_1} = 0$ for all $p \in \ZZ$.
  The result then follows from \zcref{eqn:PosSpecFlowFirstDecomp} and another application of \tkf.
\end{proof}

\subsection{Reduction with negative spectral flow}\label{subsec:SpecFlowTypeARank1Neg}

As noted in \eqref{eqn:SpecFlowDiffAction}, when $\ell>0$, the action of the differential $\diff$ on the generating state $\specfl{-\ell} \brac[\big]{ \genvec{m}{\mu} }$ is non-zero, wreaking havoc with our cohomological computations.
Happily, this can be remedied by simply changing the generating state for the fermionic ghost vertex algebra $\bcvoa$.
(Recall that $\bcvoa$ is simple, so every non-zero state is a generator.)
Given $\ell>0$, we choose the following generator for $\bcvoa$:
\begin{equation} \label{eqn:new-ghost-vacuum}
  \varphi^*_{-\ell} \cdots \varphi^*_{-2} \varphi^*_{-1} \bcvac.
\end{equation}
We then define the following generating states for the BRST complex with negative spectral flow:
\begin{equation}\label{eqn:defBRSTGenStatesSpecFlow2}
  \specfl{-\ell} \brac[\big]{ \genvec{m}{\mu} }
  = m \otimes \specflvoa{-\ell} \brac{ e^{-a + \mu c} } \otimes \varphi^*_{-\ell} \cdots \varphi^*_{-2} \varphi^*_{-1} \bcvac, \quad \text{for}\ \ell>0,\ m \in \mbasis\ \text{and}\ \mu \in [\lambda].
\end{equation}
Unlike the generators that we used for non-negative spectral flows, these have non-zero ghost grade (in fact, their grade is $\ell$).
On the other hand, these states are annihilated by the BRST differential $\diff$.
\begin{lemma}\label{lem:BRSTDiffActionNewGen}
  For all $\ell>0$, $m \in \mbasis$ and $\mu \in [\lambda]$, we have $\diff \specfl{-\ell} \brac[\big]{ \genvec{m}{\mu} } = 0$.
\end{lemma}
\begin{proof}
  Since $\specflvoa{-\ell} \brac{ e^{-a + \mu c} }$ is annihilated by the $e^c_n$, with $n>\ell$, and $\varphi^*_{-\ell} \cdots \varphi^*_{-2} \varphi^*_{-1} \bcvac$ is annihilated by the $\varphi^*_{-n}$, with $n \le \ell$ (which includes $n=0$), this is a straighforward computation:
  \begin{equation}
    \diff \specfl{-\ell} \brac[\big]{ \genvec{m}{\mu} }
    = \sum_{n \in \ZZ} m \otimes (e^c_n + \wun \kdelta{n}{0}) \specflvoa{-\ell} \brac{ e^{-a + \mu c} } \otimes \varphi^*_{-n} \varphi^*_{-\ell} \cdots \varphi^*_{-2} \varphi^*_{-1} \bcvac = 0. \qedhere
  \end{equation}
\end{proof}
\begin{remark}
	The generating state \eqref{eqn:new-ghost-vacuum} may be interpreted as the image of $\bcvac$ under a spectral flow map, see \cite{CreUni19}.
	We will not need this interpretation for what follows.
\end{remark}

We are now ready to calculate the \qhr\ of a negatively spectrally flowed \fr\ $\usltwo$-module $\IH^{-\ell}(\Mod{M})$, where $\Mod{M}$ is a $\uvir$-module, $[\lambda] \in \CC/\ZZ$ and $\ell>0$.
A convenient basis for the BRST complex is
\begin{equation}\label{eqn:NegSpecBRSTBasis}
  \setbar[\big]{ \wtd_{-\alpha} e^c_{-\beta+\ell} \varphi_{-\kappa+1+\ell} \varphi^*_{-\nu-\ell} \specfl{-\ell} \brac[\big]{\genvec{m}{\mu}} }{ \alpha, \beta \in \partitions, \ \kappa, \nu \in \partitionsFerm, \ m \in \mbasis, \ \mu \in [\lambda] },
\end{equation}
where the generating states $\specfl{-\ell} \brac[\big]{\genvec{m}{\mu}}$ are those defined in \eqref{eqn:defBRSTGenStatesSpecFlow2}.

As usual, we start the cohomology calculation by writing the explicit action of the differential $\diff$ on one of our basis vectors.
The result is similar, but not identical, to that of \zcref{lem:PosSpecFlowFirstDecomp}.
\begin{lemma} \label{lem:NegSpecFlowFirstDecomp}
  The action of the BRST differential on a basis vector $v$ of the form given in \eqref{eqn:NegSpecBRSTBasis} is
  \begin{equation}\label{eqn:diffActionExpand1Pos}
  \diff v = v_1 + v_2 + v_3 + v_4
  \end{equation}
  where
  \begin{equation}\label{eqn:diffActionExpand1PosTermsNegSpec}
    \begin{aligned}
      v_1 &= \:\sum_{\mathclap{1 \le i \le \lenp{\alpha}}}\:
      \brac[\big]{ \wtd_{-\alpha_1} \cdots \omitmode{\wtd_{-\alpha_i}} \varphi^*_{-\alpha_i} \cdots \wtd_{-\alpha_{\lenp{\alpha}}} }
      e^c_{-\beta+\ell} \varphi_{-\kappa+1+\ell} \varphi^*_{-\nu-\ell} \specfl{-\ell} \brac[\big]{ \genvec{m}{\mu} },
      \\
      v_2 &= \:\sum_{\mathclap{\substack{1 \le i \le \lenp{\kappa} \\ \kappa_i \ne 1}}}\:
      (-1)^{i-1} \wtd_{-\alpha} e^c_{-\beta+\ell}
      \brac[\big]{ \varphi_{-\kappa_1+1+\ell} \cdots \omitmode{\varphi_{-\kappa_i+1+\ell}} e^c_{-\kappa_i+1+\ell} \cdots \varphi_{-\kappa_{\lenp{\kappa}}+1+\ell} }
      \varphi^*_{-\nu-\ell} \specfl{-\ell} \brac[\big]{ \genvec{m}{\mu} },
      \\
      v_3 &= (-1)^{i-1} \wtd_{-\alpha} e^c_{-\beta+\ell}
      \brac[\big]{ \varphi_{-\kappa_1+1+\ell} \cdots \omitmode{\varphi_{-\kappa_i+1+\ell}} \cdots \varphi_{-\kappa_{\lenp{\kappa}}+1+\ell} }
      \varphi^*_{-\nu-\ell} \specfl{-\ell} \brac[\big]{ \genvec{m}{\mu} } \kdelta{\kappa_i}{\ell+1},
      \\
      v_4 &= (-1)^{\lenp{\kappa}-1} \wtd_{-\alpha} e^c_{-\beta+\ell}
      \brac[\big]{ \varphi_{-\kappa_1+1+\ell} \cdots \omitmode{\varphi_{-\kappa_{\lenp{\kappa}}+1+\ell}} }
      \varphi^*_{-\nu-\ell} \specfl{-\ell} \brac[\big]{\genvec{m}{\mu+1}} \kdelta{\kappa_{\lenp{\kappa}}}{1}.
    \end{aligned}
  \end{equation}
\end{lemma}
\noindent While $v_3$ and $v_4$ are already linear combinations of the basis elements in \eqref{eqn:NegSpecBRSTBasis}, this is not true for $v_1$ and $v_2$.
This can be corrected by reordering modes using commutation relations, but this is unnecessary for what follows.

We next forget the $\VOA{V}$-module structure on $C$ so that $\brac{C, \diff}$ is a differential complex of graded vector spaces.
Step~\ref{it:gauge} now instructs us to factorise $C$, so we introduce
\begin{equation}\label{eqn:NegSpecFlowInitialDecomp}
	\begin{aligned}
		C_0 &= \cspn \setbar[\big]{ \wtd_{-\mu} \varphi^*_{-\nu-\ell} \dket{0} }{ \mu \in \partitions, \ \nu \in \partitionsFerm}
		\\ \text{and} \quad
		C_1 &= \cspn \setbar[\big]{ e^c_{-\alpha+\ell} \varphi_{-\beta+1+\ell} \dket{\mu} }{ \alpha \in \partitions, \ \beta \in \partitionsFerm, \ \mu \in [ \lambda ] },
	\end{aligned}
\end{equation}
where $\dket{0}$ is a vector with ghost grade $\ell$ and the $\dket{\mu}$, $\mu \in [\lambda]$, are vectors of ghost grade $0$.
These subspaces $C_i$, for $i=0, 1$, are equipped with differentials $\diff_i$ defined as per \eqref{eqn:DiffDecompFactors}.
Equipping $\Mod{M}$ with the zero differential, the same argument that we used for \zcref{lem:NoSpecFirstDecomp} gives $C \cong \Mod{M} \otimes C_0 \otimes C_1$, hence
\begin{equation} \label{eqn:NegSpecFlowFirstDecomp}
  \cohom{n}{C, \diff} \cong \bigoplus_{p+q=n} \Mod{M} \otimes \cohom{p}{C_0, \diff_0} \otimes \cohom{q}{C_1, \diff_1}, \quad \text{for all } n \in \ZZ.
\end{equation}

The aim is now to show that $\brac{C_1, \diff_1}$ has trivial cohomology.
First, we factorise $C_1$ into an infinite number of differential complexes labelled by the mode index $n \in \ZZ_{\le \ell}$:
\begin{equation}\label{eqn:NegSpecFlowModeDecomp}
	C_{1,n} =
	\begin{cases*}
		\cspn \setbar[\big]{ \varphi^s_n \specfl{-\ell} \brac[\big]{\dket{\mu}} }{ s \in \set{0, 1}, \mu \in [\lambda] } & for $n=\ell$,
		\\
		\cspn \setbar[\big]{ (e^c_n)^r \varphi^s_n \dket{0} }{ r \in \ZZ_{\ge 0}, \ s \in \set{0, 1} } & for $n<\ell$.
	\end{cases*}
\end{equation}
The differential $\diff_{1,n}$ on $C_{1,n}$ is given by restricting $D_1$ in the obvious way.
Computing
\begin{equation}
	\diff_{1,\ell} \varphi_{\ell} \specfl{-\ell} \brac[\big]{\dket{\mu}} = \specfl{-\ell} \brac[\big]{\dket{\mu+1}}.
\end{equation}
using \zcref{lem:NegSpecFlowFirstDecomp}, the computation proceeds as for \zcref{lem:PosSpecTrivialFactor}.
\begin{lemma}\label{lem:NegSpecTrivialFactor}
  For all $p \in \ZZ$, we have $\cohom{p}{C_{1,\ell}, \diff_{1,\ell}} = 0$.
\end{lemma}
\begin{theorem}\label{thm:NegativeSpecFlowCompositionsltwo}
	For all $\ell>0$, $\uvir$-modules $\Mod{M}$ and $[\lambda] \in \CC/\ZZ$, we have
	\begin{equation}
		\brac{\QH \circ \IH^{-\ell}} \brac{\Mod{M}} = 0.
	\end{equation}
\end{theorem}

\section{Applications to the category of weight $\ssltwo$-modules at admissible level}\label{sec:AdmissLevel}

To consider some applications of \zcref{thm:sltwocomposition, thm:PositiveSpecFlowCompositionsltwo, thm:NegativeSpecFlowCompositionsltwo}, we first recall the representation theory of $\ssltwo$ at admissible but non-integral level.
This case has recently attracted considerable attention in the literature, as the representation theory is intricate yet sufficiently constrained that understanding it is a feasible task.

\subsection{Weight $\ssltwo$-modules at admissible levels}\label{sec:LogMods_review}

For non-critical levels $\kk$, the \voas\ $\usltwo$ and $\uvir$ have unique simple quotients denoted by $\ssltwo$ and $\svir$, respectively.
In particular, when $\kk=-2+\frac{\uu}{\vv}$ is admissible, see \eqref{eqn:admissible_condition}, the simple quotients are sometimes denoted by $\sltwomm$ and $\virmm$ to emphasise the role of the parameters $\uu$ and $\vv$.

At admissible but non-integral levels, there are a finite number of simple $\virmm$-modules $\virirred$ parameterised by $1 \le r \le \uu-1$ and $1 \le s \le \vv-1$ \cite{Wan93,RidJac14}.
They are precisely the simple \hw\ $\uvir$-modules whose conformal weights have the form
\begin{equation}\label{eqn:VirConfWeightDef}
	\virhwt = \frac{\brac{\vv r- \uu s}^2 - \brac{\uu - \vv}^2}{4\uu \vv}.
\end{equation}
These modules are all mutually inequivalent, except that $\virirred \cong \virirred[\uu-r,\vv-s]$.

To describe the corresponding situation for $\sltwomm$, it will be useful to set
\begin{equation}\label{eqn:sltwoWeightDef}
	\sltwohwt = r-1-\frac{\uu}{\vv}s \quad \text{and} \quad
	\sltwoconfhwt = \frac{\brac{\vv r- \uu s}^2 - \vv^2}{4\uu \vv}.
\end{equation}
The simple \rhw\ $\sltwomm$-modules may be naturally sorted into four classes \cite{AM95, RW15}, parameterised by $r=1, \dots, \uu-1$, $s=1, \dots, \vv-1$ and $[\sltwowt] \in \CC / 2\ZZ$, with $[\sltwowt]\ne[\sltwohwt], [\sltwohwt[\uu-r, \vv-s]]$.
Up to isomorphism, these consist of:
\begin{itemize}
	\item Ordinary modules $\sltwoqfin$: the \hwv\ has $h_0$- and $T^{\sltwo}_0$-eigenvalues $\sltwohwt[r,0]$ and $\sltwoconfhwt[r,0]$, respectively.
	\item Highest-weight modules $\sltwohwm$: the \hwv\ has $(h_0, T^{\sltwo}_0)$-eigenvalues $(\sltwohwt, \sltwoconfhwt)$.
	\item Conjugate-\hwms\ $\sltwolwm$: the \chwv\ has eigenvalues $(-\sltwohwt, \sltwoconfhwt)$.
	\item Fully relaxed modules $\sltwoirrhwm$: the \rhwvs\ have eigenvalues $(\Lambda+2n, \sltwoconfhwt)$, $n \in \ZZ$.
\end{itemize}
These simples are all mutually inequivalent, except that $\sltwoirrhwm \cong \sltwoirrhwm[{[\Lambda]};\uu-r,\vv-s]$.

For each $1 \le r \le \uu-1$ and $1 \le s \le \vv-1$, there also exist \fr\ $\sltwomm$-modules $\sltworhwm^{\pm} = \sltwoirrhwm[{[\sltwohwt]};r,s]^{\pm}$ \cite{CR13,AdaRea17,KR19}.
They are determined up to isomorphism by the following non-split short exact sequences:
\begin{equation}\label{eqn:reducibleButIndecompType}
	\ses{\sltwohwm}{\sltworhwm^+}{\sltwolwm[\uu-r, \vv-s]}
	\quad \text{and} \quad
	\ses{\sltwolwm}{\sltworhwm^-}{\sltwohwm[\uu-r, \vv-s]}.
\end{equation}

It was shown in \cite{AdaRea17} that the inverse-reduction functors $\IH$ construct \fr\ $\sltwomm$-modules.
Specifically, for all $r=1,\dots,\uu-1$ and $s=1,\dots,\vv-1$, the image $\IH(\virirred) = \virirred \otimes \wlatt$ is \fr.
Its top space is spanned by the $v_{r,s} \otimes \ee^{-a+\mu c}$, with $\mu \in [\lambda]$, where $v_{r,s}$ denotes a \hwv\ of $\virirred$.
From the embedding \eqref{eqn:sltwoIQHRembedding}, it is easy to show that these top space vectors have $h_0$-eigenvalue $2\mu$ and conformal weight $\virhwt + \frac{\kk}{4} = \sltwoconfhwt$.
Moreover, \eqref{eqn:sltwoIQHRembedding} also shows that the action of $e_0$ is injective.
This suggests that
\begin{equation}\label{eqn:simpleSLTwoFullyRelaxedViaIQHR}
	\IH(\virirred) \cong
	\begin{cases*}
		\sltworhwm^- & if $[\lambda] = [\frac{1}{2}\sltwohwt[\uu-r,\vv-s]]$,\\
		\sltworhwm[\uu-r,\vv-s]^- & if $[\lambda] = [\frac{1}{2}\sltwohwt]$, \\
		\sltwoirrhwm & otherwise.
	\end{cases*}
\end{equation}
The suggestion is correct, though the original proof \cite{AdaRea17} relied on a conjectural character formula for the $\sltwoirrhwm$ that was subsequently proven in \cite{KR19}.
A proof avoiding characters eventually appeared in \cite[Rem.~6.5]{AKR21}.
The minus reductions of these modules may be obtained immediately by combining \zcref{thm:sltwocomposition, thm:PositiveSpecFlowCompositionsltwo, thm:NegativeSpecFlowCompositionsltwo} with \eqref{eqn:simpleSLTwoFullyRelaxedViaIQHR}.

\begin{corollary}\label{cor:NegRedFullyRelaxedWtCat}
  The minus reductions of the simple \frms\ $\sltwoirrhwm[\sltwowt; r, s]$ and the indecomposable \frms\ $\sltworhwm^{-}$, along with their spectral flow twists, are given by
  \begin{equation}
    \nQH^n \brac*{ \specfsltwo{\ell} \brac{\sltwoirrhwm[\sltwowt; r, s]} } \cong
    \nQH^n \brac*{ \specfsltwo{\ell} \brac{\sltworhwm^{-}} } \cong \kdelta{n}{0} \kdelta{\ell}{0} \virirred.
  \end{equation}
\end{corollary}

\subsection{Non-semisimple projective $\ssltwo$-modules}\label{sec:LogMods_projectives}

Next, we consider the projective cover $\sltwoproj^{\ell} = \specfsltwo{\ell} \brac{\sltwoproj}$ of the \hw\ $\sltwomm$-module $\specfsltwo{\ell} (\sltwohwm)$, where $\ell\in\ZZ$, $1 \le r \le \uu-1$ and $1 \le s \le \vv-1$ (and $\vv>1$).
Because the $\sltwolwm$ and $\sltwoqfin$ are spectral flows of some $\sltwohwm[r',s']$, we do not need to consider their projective covers separately.
In \zcref{sec:LogMods_QHR}, we will compute the \qhrs\ of the $\sltwoproj^{\ell}$.

\begin{proposition}\label{prop:ssltwoRedButIndecompProjCovers}
  For $\vv>1$, there exist indecomposable $\sltwomm$-modules $\sltwoproj^{\ell}$, for $\ell \in \ZZ$, $1 \le r \le \uu-1$ and $1 \le s \le \vv-1$, whose structures are characterised by the following non-split short exact sequences:
  \begin{gather}
	  \ses{\specfsltwo{\ell+1} (\sltworhwm[\uu-r,\vv-1-s]^-)}{\sltwoproj^{\ell}}{\specfsltwo{\ell} (\sltworhwm[\uu-r,\vv-s]^-)}
	  \qquad (s \ne \vv-1), \label{eqn:ssltwoRedButIndecompProjCoverCaseI}
	  \\
	  \ses{\specfsltwo{\ell+2} (\sltworhwm[r,\vv-1]^-)}{\sltwoproj[r,\vv-1]^{\ell}}{\specfsltwo{\ell} (\sltworhwm[\uu-r,1]^-)}.
	  \qquad \hphantom{(s \ne \vv-1)} \label{eqn:ssltwoRedButIndecompProjCoverCaseII}
  \end{gather}
  These modules are pairwise non-isomorphic.
\end{proposition}
\noindent Examples of these modules were first discussed in \cite{GabFus01,RidFus10}.
The first general result appeared in \cite{AdaRea17}, where the modules defined by \eqref{eqn:ssltwoRedButIndecompProjCoverCaseI} were constructed.
An alternative approach, which also included those of \eqref{eqn:ssltwoRedButIndecompProjCoverCaseII}, was subsequently described in \cite{ACK23}.
In the latter work, the projectivity of these modules was also proved.

In order to compute the reductions of the projective covers $\sltwoproj^{\ell}$, we briefly review the construction of \cite{AdaRea17} (which is based on \cite{FFHST02, AM07, AM09}).
We will also simplify matters by restricting to $s\ne\vv-1$, hence to the exact sequence \eqref{eqn:ssltwoRedButIndecompProjCoverCaseI}.
Note that this restriction requires that $\vv>2$.

Let $\VOA{W} \subset \vvoa$ be a conformal embedding of \voas.
We suppose that there exists a $\vvoa$-module $\Mod{M} \ne \vvoa$, an element $\scrn{S} \in \Mod{M}$, and an intertwining map $\scrn{Y}(\blank, z)$ of type $\binom{\Mod{M}}{\Mod{M} \ \vvoa}$ for which the singular part of the \ope\ $\scrn{Y}(\scrn{S}, z) W(w)$ is a total $z$-derivative, for all $W \in \VOA{W}$.
This forces $\scrn{S}$ to be, among other things, a Virasoro \hwv\ of conformal weight $1$.
Its zero mode $\scrn{S}_0$ is called a \emph{screening operator} and it commutes with the action of $\VOA{W}$.
More precisely, it defines a $\VOA{W}$-module homomorphism between the $\vvoa$-modules $\Mod{N}$ and $\Mod{N}' \subseteq \Mod{M}\boxtimes\Mod{N}$ when the intertwining map $\scrn{S}(z) = \scrn{Y}_{\Mod{N}} (\scrn{S}, z)$ of type $\binom{\Mod{N}'}{\Mod{M} \ \Mod{N}}$ acts on $\Mod{N}$ with trivial monodromy around $0$.

We make a couple of simplifying remarks.
First, the requirement on the \opes\ involving $\scrn{Y}(\scrn{S}, z)$ will be satified for all $W \in \VOA{W}$ if it is satisfied for a set of strong generators.
Second, if we have found a Virasoro \hwv\ $\scrn{S} \in \Mod{M}$ of conformal weight $1$, then this last requirement is equivalent to the \opes\ having no simple pole.
In other words, it is enough to check that $\scrn{S}_0 W = 0$ for each strong generator $W \in \VOA{W}$ (hence that $\VOA{W} \subseteq \ker \scrn{S}_0$).

The Semikhatov embedding \eqref{eqn:sltwoIQHRembedding} descends to an embedding of simple quotients when $\vv>1$ \cite{AdaRea17}.
We may thus take $\vvoa = \svir \otimes \lvoa$ and $\VOA{W} = \ssltwo = \sltwomm$ in the above setup.
To determine a good screening operator, we shall use the fact that the projective module in \eqref{eqn:ssltwoRedButIndecompProjCoverCaseI} will be constructed by ``glueing'' the submodule and quotient together using the action of this operator (in the final result, it acts as the nilpotent part of $T^{\sltwo}_0$).
We therefore let
\begin{equation}
	\Mod{N}' = \specfsltwo{\ell+1} (\sltworhwm[\uu-r,\vv-1-s]^-) \cong \virirred[r,s+1] \otimes \specflatt[\sltwohwt[r,s+1]/2][\ell+1]
	\quad \text{and} \quad
	\Mod{N} = \specfsltwo{\ell} (\sltworhwm[\uu-r,\vv-s]^-) \cong \virirred \otimes \specflatt[\sltwohwt/2],
\end{equation}
(the isomorphisms follow from comparing weights and conformal weights using \eqref{eqn:sltwoIQHRembedding}) and demand that $\Mod{M} \boxtimes \Mod{N}$ contains $\Mod{N}'$ as a submodule.

As $\lvoa$ is a lattice \voa, it is easy to determine its fusion rules:
\begin{equation}
	\specflatt \boxtimes \specflatt[\mu][m] \cong \specflatt[\lambda+\mu+\kk/2][\ell+m-1], \quad \ell,m \in \ZZ,\ [\lambda],[\mu] \in \CC/\ZZ.
\end{equation}
Combining this with the well known fusion rules of the Virasoro minimal models, the simplest candidate module satisfying our demand is thus
\begin{equation}
	\Mod{M} = \virirred[1,2] \otimes \specflatt[-\uu/\vv][2]
\end{equation}
(the assumption that $\vv>2$ guarantees that $\virirred[1,2]$ is a module of $\svir$).
Moreover, $\Mod{M}$ contains a \hwv\ of conformal weight $1$ (with respect to $T^{\sltwo}$), namely
\begin{equation}
	\scrn{S} = v_{1,2} \otimes \ee^a \in \Mod{M}.
\end{equation}
Here, $v_{1,2}$ is a \hwv\ of $\virirred[1,2]$.
We also note that $\scrn{S}$ has $h_0$-eigenvalue $0$, matching that of the nilpotent part of $T^{\sltwo}_0$.

What does not (yet) match is the order of nilpotence of $\scrn{S}_0$ and $T^{\sltwo}_0$.
The latter is expected to have order $2$, but the former appears not to be nilpotent at all.
To fix this, as well as implement the ``glueing'' of $\Mod{N}'$ and $\Mod{N}$, we extend $\vvoa = \svir \otimes \lvoa$ to the semidirect sum $\bvvoa = \vvoa \loplus \Mod{M}$ using the intertwining map $\scrn{S}(z) = \scrn{Y}_{\Mod{N}} (\scrn{S}, z)$ of type $\binom{\Mod{N}'}{\Mod{M} \ \Mod{N}}$ (see \cite{Li94Sym} for the definition and basic properties).
We equip $\bvvoa$ with the same \emt\ as $\vvoa$ and note that this semidirect structure means that the self-\ope\ of $\scrn{S}(z)$ is identically zero in $\bvvoa$.
It therefore follows that $\scrn{S}_0$ is nilpotent of order (at most) $2$.
(In fact, $\scrn{S}_m \scrn{S}_n = 0$ for all $m,n \in \ZZ$.)

The glueing is now simply realised by inducing $\Mod{N}$ to a $\bvvoa$-module $\bppmod = \bvvoa \boxtimes \Mod{N}$.
As a $\vvoa$-module, this is isomorphic to $\Mod{N}' \oplus \Mod{N}$.
However, as a $\bvvoa$-module, we can act with the modes of elements like $\scrn{S}$.
Explicitly, the $\bvvoa$-action of $v + m$, where $v \in \vvoa$ and $m \in \Mod{M}$, on $n' \oplus n$, where $n' \in \Mod{N}'$ and $n \in \Mod{N}$, is given by
\begin{equation}\label{eqn:semiDirectSumModStructure}
	Y_{\bppmod}(v+m,z) (n'+n) = Y_{\Mod{N}'}(v,z) n' + Y_{\Mod{N}}(v,z) n + \scrn{Y}_{\Mod{N}}(m,z) n.
\end{equation}

We can therefore now verify that $\scrn{S}_0 \in \bvvoa$ is a screening operator for $\VOA{W} = \sltwomm$.
From the above discussion, we only need check that the strong generators $e$, $h$ and $f$ are annihilated by $\scrn{S}_0$.
For $e$ and $h$, this is trivially verified because their \opes\ with $\scrn{S}$ are regular.
This is not the case for $f$, but the computation is nevertheless straightforward, if a little tedious.

The final step is to twist the action on the $\bvvoa$-module $\bppmod$ using the spectral flow map induced by $\scrn{S}$.
More precisely, we will change the action by inserting $\Delta(\scrn{S},z)$ as in \eqref{eq:specflow}.
The result might no longer be a $\bvvoa$-module, because $\scrn{S}_0$ acts non-semisimply on $\bvvoa$, but it will be a module for $\ker \scrn{S}_0$ and hence for its subalgebra $\VOA{W} = \sltwomm$ \cite{AM07}.
This restricted module is $\sltwoproj^{\ell}$ and the action is explicitly given by
\begin{equation}\label{eqn:logMod_TwistedAction}
  Y_{\sltwoproj^{\ell}} \brac{A,z} = Y_{\bppmod} \brac[\big]{\Delta \brac{\scrn{S}, z} A, z}, \quad \text{for}\ A \in \sltwomm,
\end{equation}
where $\Delta$ is defined in \eqref{eqn:LiSpecFlow}.
It is straightforward to verify that $T^{\sltwo}_0$ acts on $\sltwoproj^{\ell}$ in the same way that $T^{\sltwo}_0 - \scrn{S}_0$ acts on $\bppmod$.
This action is therefore non-semisimple with Jordan blocks of rank at most $2$, as expected.

\subsection{Reduction of non-semisimple projective $\ssltwo$-modules}\label{sec:LogMods_QHR}

We now consider the minus reduction of the projective modules $\sltwoproj^{\ell}$.
By definition, this is the BRST cohomology of the complex $C(\sltwoproj^{\ell}) = \sltwoproj^{\ell} \otimes \bcvoa$ with differential given by the zero mode of $\cBRST(z) = \brac[\big]{\ee^c(z) + \wun} \otimes \varphi^*(z)$, see \zcref{sec:CompositionTypeARankOne}.
The action of $\diff = \cBRST_0$ on $C(\sltwoproj^{\ell})$ matches the action of the zero mode of the $\scrn{S}$-twisted field $Y_{\bppmod}\brac[\big]{\Delta(\scrn{S},z) \cBRST, z}$ on $C(\bppmod) = \bppmod \otimes \bcvoa$.
We therefore analyse this latter action.

Since $\scrn{S}(z) = v_{1,2}(z) \otimes \ee^a(z)$ and $\ee^c(w)$ have regular \ope, we have $S_n \ee^c = 0$ for all $n\ge0$.
It follows that
\begin{equation}
  \Delta(\scrn{S},z) \brac[\big]{\ee^c+\ket{0}}
  = \sqbrac[\bigg]{z^{-\scrn{S}_0} \prod_{n=1}^{\infty} \exp \brac*{ \frac{(-z)^{-n}}{n} \scrn{S}_n } \brac[\big]{\ee^c+\ket{0}}}
  = \ee^c+\ket{0},
\end{equation}
using \eqref{eqn:LiSpecFlow}.
The action \eqref{eqn:logMod_TwistedAction} of $\cBRST$ on $C(\sltwoproj^{\ell})$ is thus given by
\begin{equation}
  \begin{split}
    Y_{C(\sltwoproj^{\ell})} (\cBRST, z)
    &= Y_{\sltwoproj^{\ell}} \brac[\big]{\ee^c + \ket{0}, z} \otimes \varphi^*(z)
    = Y_{\bppmod} \brac[\Big]{\Delta \brac{\scrn{S}, z} \brac[\big]{\ee^c + \ket{0}}, z} \otimes \varphi^*(z) \\
    &= Y_{\bppmod} \brac[\big]{\ee^c + \ket{0}, z} \otimes \varphi^*(z).
  \end{split}
\end{equation}
We apply this result to $(n'+n) \otimes u \in C(\sltwoproj^{\ell})$ (so $n' \in \Mod{N}'$, $n \in \Mod{N}$ and $u \in \bcvoa$).
Because $\ee^c+\ket{0}$ belongs to $\vvoa = \uvir \otimes \lvoa$, the result is
\begin{equation}
  \begin{split}
    Y_{C(\sltwoproj^{\ell})} (\cBRST, z) \sqbrac[\big]{(n'+n) \otimes u}
    &= \sqbrac[\big]{Y_{\bppmod} \brac[\big]{\ee^c + \ket{0}, z} \otimes \varphi^*(z)} \sqbrac[\big]{(n'+n) \otimes u} \\
    &= Y_{\Mod{N}'} \brac[\big]{\ee^c + \ket{0}, z} n' \otimes \varphi^*(z) u + Y_{\Mod{N}} \brac[\big]{\ee^c + \ket{0}, z} n \otimes \varphi^*(z) u \\
    &= Y_{C(\Mod{N}')} (Q, z) (n' \otimes u) + Y_{C(\Mod{N})} (Q, z) (n \otimes u) \\
    &= Y_{C(\Mod{N}' \oplus \Mod{N})} (Q, z) \brac[\big]{(n'+n) \otimes u},
  \end{split}
\end{equation}
by \eqref{eqn:semiDirectSumModStructure}.
In other words, the action of the differential on $C(\sltwoproj^{\ell})$ matches its action on $C(\Mod{N}' \oplus \Mod{N})$.

The corresponding BRST reductions are therefore identical as graded vector spaces.
But, these reductions belong to a semisimple category in which the irreducibles are completely determined by the grading (that is, by the characters).
We therefore obtain our final result, courtesy of \zcref{cor:NegRedFullyRelaxedWtCat}.
\begin{theorem}\label{thm:QHRProjCoversPos}
  For any $\ell \in \ZZ$, $1 \le r \le \uu-1$ and $1 \le s \le \vv-2$, the minus reduction of $\specfsltwo{\ell} \brac{\sltwoproj}$ is given by
  \begin{equation}
  \nQH^n \brac[\big]{\specfsltwo{\ell} \brac{\sltwoproj}} \cong
    \begin{cases*}
      \kdelta{n}{0} \virirred & for $\ell = 0$, \\
      \kdelta{n}{0} \virirred[r,s+1] & for $\ell = -1$, \\
      0 & otherwise.
    \end{cases*}
  \end{equation}
\end{theorem}


\appendix

\section{Spectral sequences} \label{sec:UnboundedSpecSequences}

A common strategy to compute the cohomology of a complex $\brac{ C^{\bullet}, \diff }$ is to introduce a filtration $F$ and construct the corresponding \emph{filtered spectral sequence} $\brac{ E_R^{p, q}, \diff_R }$.
Here, we review some standard definitions and convergence results for spectral sequences.
This includes less well known phenomena concerning the convergence of unbounded spectral sequences, following \cite{Boa99,Wei94}.

\begin{definition}\label{def:compatibleFilt}
	Given a filtration $F$ on a complex $\brac{ C^{\bullet}, \diff }$, we say that the differential is \umph{compatible} with the filtration if $\diff \brac{F^p C^n} \subseteq F^p C^{n+1}$.
\end{definition}
\noindent From here on, we assume that the filtration $F^p C$ is decreasing, indexed by $p \in \ZZ_{\ge 0}$ and compatible with the differential.
The indexing is simply to match our intended application (Li's filtration --- see \zcref{def:LiFiltrationVOAMod}).

\begin{definition}
We say that $F$ is a \umph{bounded filtration} of $C$ if $F^p C = 0$ for all $p \ge p_{\text{max}}$.
If $F$ is not bounded, then it is called an \umph{unbounded filtration}.
\end{definition}
\noindent For unbounded spectral sequences, reconstructing the cohomology from the limiting page $E_{\infty}$ of a spectral sequence is a significantly complicated and subtle task.

Suppose that $\iota^{p, n} \colon F^p C^n \to C^n$ is the inclusion map, which induces a map $\iota^{p, n}_*$ on the cohomology.
Then, there is an induced filtration $\widetilde{F}$ on the cohomology $H = \cohom{}{C, \diff}$, defined by
\begin{equation}\label{eqn:inducedFilt}
	\widetilde{F}^p \cohom{n}{C, \diff} = \im \brac*{ \iota^{p, n}_* \colon \cohom{n}{F^p C, \diff} \to \cohom{n}{C, \diff} }.
\end{equation}
Given a filtration of $C$, the associated graded space is defined to be
\begin{equation}
    \Gr C = \bigoplus_{p \in \ZZ_{\ge 0}} \Gr^p C,
    \quad \text{where} \quad
    \Gr^p C = F^p C \big/ F^{p+1} C.
\end{equation}

Now, recall the construction of the filtered spectral sequence from a filtration $F$ on $\brac{ C^{\bullet}, \diff }$.
First, the grading $C^{\bullet}$ on the complex is refined to a bigrading by setting
\begin{equation}
	C^{p, q} = \:\bigoplus_{\mathclap{p \in \ZZ_{\ge 0}, \ q \in \ZZ}}\: F^p C^{p+q}.
\end{equation}
We call the original grading $n=p+q$ the \emph{total grading}, $p$ the \emph{filtered grading} and $q$ the \emph{complementary grading}.

The initial page of the filtered spectral sequence is denoted by $E_0$.
It has terms defined by
\begin{equation}
	E_0^{p, q} = \Gr^p C^{p+q} = \frac{F^p C^{p+q}}{F^{p+1} C^{p+q}}.
\end{equation}
Let $\diff_0$ be the differential induced by $\diff$ on the associated graded space $E_0$.
The pair $\brac[\big]{E_0^{p, q}, \diff_{0}}$ then defines a differential complex.
Its cohomology is the page $E_1$ and it may be equipped with a differential $\diff_1$.

More generally, if we are given a page $\brac{ E^{p, q}_R, \diff_R }$, then the subsequent page is defined to be the cohomology
\begin{equation}
	E_{R+1}^{p, q} = \cohom{}{E_R^{p, q}, \diff_R},
\end{equation}
equipped with a differential $\diff_{R+1}$.
The differentials $\diff_R$, for $R \ge 1$, are difficult to construct in general, but always act to change the bidegree of a homogeneous element by $\brac{R, 1-R}$.
In this paper, all spectral sequences turn out to collapse at $E_1$, so we will not need to discuss these higher differentials any further.

A spectral sequence $(E^{p, q}_R, \diff_R)$ is called \emph{half plane} if there exists $N \in \ZZ$ so that either $E^{p, q}_R = 0$ for all $p<N$, or $E^{p, q}_R = 0$ for all $p>N$.
As we will generally work with filtrations indexed by $\NN$, every spectral sequence in this paper will be half plane.
\begin{definition}\label{def:limitingTerms}
	If there exists some $R_{\max}$ such that $E^{p, q}_R \cong E^{p, q}_{R_{\max}}$, for all $R \ge R_{\max}$, then we say that the spectral sequence has \umph{collapsed} and has \umph{limiting page} $E_{\infty}^{p, q} = E^{p, q}_{R_{\max}}$.
\end{definition}
\noindent When a spectral sequence collapses, we may identify its limiting page with the associated graded space of some graded object $V$ as follows:
\begin{equation}\label{eqn:compareAssGraded}
	\Gr^p V^{p+q} = E^{p, q}_{\infty}
\end{equation}
where $\Gr V$ is the associated graded space of $V$ with respect to some filtration $\fV$.
We typically desire that $V^{\bullet}$ is isomorphic to the cohomology $H^{\bullet} = \cohom{}{C, \diff}$.
In this case, the filtration $\fV$ of $V$ corresponds to the induced filtration $\widetilde{F}$ on cohomology, see \eqref{eqn:inducedFilt}.
Unfortunately, without additional conditions being met, it may not be possible to identify $V^{\bullet}$ with the cohomology.

\begin{definition}
	When a spectral sequence collapses, we say that it \umph{weakly converges} to some graded object $V^{\bullet}$ if:
	\begin{itemize}
		\item $V^{\bullet}$ is equipped with a decreasing filtration $\cdots \supseteq \fV^{p-1} V^n \supseteq \fV^p V^n \supseteq \fV^{p+1} V^n \supseteq \cdots$, for each $n$, and
		\item there exist isomorphisms $\alpha^{p, q}\colon E_{\infty}^{p, q} \to \Gr^p V^{p+q}$, for all $p, q \in \ZZ$.
	\end{itemize}
\end{definition}
\begin{proposition}\label{prop:IsomorphicAssGraded}
	Suppose $X^{\bullet}$ is a graded object with a decreasing filtration $F^{\bullet} X$.
	Then, the graded objects $X^{\infty} = \bigcup_{p=0}^{\infty} F^p X$ and $X^{-\infty} = X / \bigcap_{p=0}^{\infty} F^p X$ may be equipped with the obvious induced filtrations coming from $C$ and the associated graded spaces for $X$, $X^{\infty}$ and $X^{-\infty}$ are all isomorphic.
\end{proposition}
\noindent Note that \zcref{prop:IsomorphicAssGraded} shows that associated graded spaces are insensitive to elements \emph{outside} $\bigcup_p F^p X^n$ and elements \emph{inside} $\bigcap_p F^p X^n$.

\begin{definition}\label{def:NiceFiltrationConditions}
	A filtration $F$ on a vector space $X$ is called \umph{exhaustive} if $X^{\bullet} = \bigcup_p F^p X^{\bullet}$ and \umph{Hausdorff} (or \umph{separated}) if $\bigcap_p F^p X^{\bullet} = \set{ 0 }$.
\end{definition}
\noindent It is important that these conditions hold because if the filtration $F$ on $X$ is not exhaustive and Hausdorff, then it is impossible to fully reconstruct $X$ from $\Gr X$.

In Boardman's work \cite{Boa99}, it was shown there are a number of obstructions that may prevent $V$ from being identifiable with the cohomology $H$.
For instance, obstructions typically arise for unbounded spectral sequences that do not collapse, or are not half plane.

\begin{theorem}[\cite{Boa99}]\label{thm:WeakConvg_HalfPlane}
	Consider a complex $(C, \diff)$ in which $C$ is equipped with a compatible decreasing filtration $F$, whose corresponding spectral sequence is $(E^{p, q}_R, \diff_R)$.
	Assume that $F$ is exhaustive and Hausdorff.
	Assume also that $(E_R, \diff_R)$ is half plane.
	The spectral sequence then converges weakly to $H = \cohom{}{C, \diff}$.
\end{theorem}
\noindent From here on, suppose that we have weak convergence to $H$.
Our next task is to determine when we can reconstruct $H$ from $\Gr H$.

More generally, if $X$ is a graded vector space with a decreasing filtration $F$, then we must consider how to reconstruct from $\Gr X$ the types of elements listed below.
For ease of reference, we assign to each family a `type' (although this is not standard terminology).
\begin{itemize}
	\item\textbf{Type I:} $x \in X$ lies in finitely many filtered components, so for $x \in X$ there exists $p_{\text{max}}$ such that $x \in F^p X$ for all $p \le p_{\text{max}}$ and $x \not\in F^p X$ for all $p>p_{\text{max}}$.
	\item\textbf{Type II:} $x \in X$ lies in \emph{all} filtered components, so $x \in F^p X$ for all $p \in \NN$.
	\item\textbf{Type III:} $x \in X$ lies in \emph{no} filtered component.
\end{itemize}
For a bounded filtration, one may encounter elements of types I and III, but the only element of type II is $0$.
For unbounded filtrations, there may also exist non-zero elements of type II.
However, if $F$ is exhaustive and Hausdorff, then $X$ has no non-zero elements of type II and no elements whatsoever of type III to reconstruct.
This motivates introducing a stronger notion of \emph{convergence}.
\begin{definition}\label{def:SpecSeqConverge}
	A spectral sequence that collapses is said to \umph{converge} to a graded object $V$ with filtration $\fV^{\bullet}V$ if it weakly converges to $V$ and the filtration $\fV$ is both exhaustive and Hausdorff.
\end{definition}
\noindent To demonstrate convergence of a filtered spectral sequence, we therefore need to show that the induced filtration on the cohomology is exhaustive and Hausdorff, but without knowing this cohomology in advance!
Half of this verification turns out to be relatively straightforward.
\begin{proposition}
	If the filtration on a complex $\brac{C^{\bullet}, \diff}$ is exhaustive, then so is the induced filtration on the cohomology $H^{\bullet}$.
\end{proposition}
\noindent Unfortunately, it is in general difficult to check if the induced filtration on the cohomology is Hausdorff.
Typically, this must be done case by case.
In \zcref{lem:LiFiltLowerBoundConfFilt_Cohom}, we verify this for the spectral sequence formed using Li's filtration, when computing the reduction of fully-relaxed $\usltwo$-modules.

Finally, we discuss reconstructing the type I elements of the cohomology.
Elements of type I are in bijective correspondence with the elements of $\Gr X$.
Care must also be taken when specifying how elements of type I are identified in $\Gr X$.
This requires solving the \emph{extension problem}: choosing an appropriate lift from $\Gr X$ to $X$.

Recall that every decreasing filtration $F$ on $X$ induces a short exact sequence
\begin{equation}\label{eqn:FiltrationSES}
	0 \to F^{p+1} X \xrightarrow{\iota^p} F^p X \xrightarrow{\pi^p} \Gr^p X \to 0
\end{equation}
where $\iota^p \colon F^{p+1} X \rightarrow F^p X$ is the inclusion map and $\pi^p \colon F^p X \to \Gr^p X$ is the canonical projection map.
A splitting is then a map $s^p \colon \Gr^p X \to F^p X$ such that $\pi^p \circ s^p = \id$.
Combining these maps to form $\bigoplus_p s^p \colon \Gr X \to X$ then gives us a map whose image is precisely the type I elements of $X$.

Equivalent choices of splittings are classified by the extension classes $\Ext (\Gr^p X, F^{p+1} X)$, so determining this is the extension problem.
For the applications in this paper, we always reconstruct $X$ from $\Gr X$ as a vector space.
In this case, all extension groups are trivial, so all splittings are equivalent up to isomorphism.
However, there is still no canonical choice of splittings $s^p$, without additional constraints to fix them.
We will not dwell on this problem here, but it will need addressing for some of the more complicated applications that we will report on in the future.

\section{Cohomology calculations}\label{sec:CohomAppendix}

There are two standard complexes that we use throughout this paper, namely the \emph{gauged lattice complex} and the \emph{Cartan complex}.
We collect here their standard cohomology calculations.

The first is a variation of the \emph{gauged complex} which appears when computing the BRST cohomology of affine vertex algebras. We briefly recall its definition.
Consider the vector space
\begin{equation}\label{eqn:gaugedComplexDef}
	A = \cspn \setbar[\big]{ \varphi_{-\alpha+1} F_{-\beta+1} \ket{0} }{ \alpha \in \partitionsFerm \ \text{and} \ \beta \in \partitions },
\end{equation}
where $\varphi(z)$ is a fermionic field and $F(z)$ is a bosonic field, both of conformal weight $0$.
Here, we utilise the notation for partitions summarised in \zcref{subsec:partitions}.
This space is graded by the \emph{ghost number}, defined so that $\ghgr \varphi_n = -1$, $\ghgr F_n = 0$ and $\ghgr \ket{0} = 0$.
We also assume that it is equipped with a fermionic differential $D$ with $\ghgr D = 1$ that satisfies
\begin{equation}\label{eqn:gaugedComplexDiffAction}
	\comm{\diff}{\varphi_n} = F_n + \wun \delta_{n, 0}, \quad
	\comm{\diff}{F_n} = 0 \quad \text{and} \quad
	\diff \ket{0} = 0, \qquad n \in \ZZ.
\end{equation}
Then, $(A,\diff)$ is a graded differential complex whose cohomology can be determined straightforwardly.
\begin{proposition} \label{prop:GaugedComplexModeCohom}
	The cohomology of the graded differential complex $\brac{A, \diff}$ is given by
	\begin{equation}
    \cohom{p}{A, \diff} \cong \CC \kdelta{p}{0}.
	\end{equation}
\end{proposition}

Roughly, the gauged lattice complex is obtained from the gauged complex by replacing the single generating state $\ket{0}$ is replaced by a family of generating states $\setbar{\ket{\mu}}{\mu \in [ \lambda ]}$, for some $[\lambda] \in \CC/\ZZ$.
The field $F(z)$ is moreover replaced by a field $e^c(z)$ satisfying $e^c_0\ket{0}=0$ and $e^c_0\ket{\mu}=\ket{\mu+1}$.
Thus, we now consider the vector space
\begin{equation}\label{eqn:gaugedLattComplexDef}
    B = \cspn \setbar[\big]{\varphi_{-\alpha+1} e^c_{-\beta} \ket{\mu} }{\alpha \in \partitionsFerm, \ \beta \in \partitions, \ \mu \in [ \lambda ] }.
\end{equation}
The ghost grading defined previously for $A$ extends to $B$ by setting $\grgh \ket{\mu} = 0$ and the differential $\diff$ satisfies, in addition to \eqref{eqn:gaugedComplexDiffAction},
\begin{equation}\label{eqn:gaugedLattComplexDiffAction}
	\diff \ket{\mu} = 0,\quad \mu \in [\lambda].
\end{equation}

To decompose our complex by mode number, we introduce
\begin{equation}
	B_n =
	\begin{cases*}
    \cspn \setbar[\big]{ (e^c_n)^r \varphi_n^s \ket{0} }{ s \in \set{ 0, 1 }, \ r \in \ZZ_{\ge 0} } ,& for $n \le -1$,
    \\
    \cspn \setbar[\big]{ \varphi_0^s \ket{\mu} }{ s \in \set{ 0, 1 }, \ \mu \in [ \lambda ] } ,& for $n = 0$.
	\end{cases*}
\end{equation}
Note that the states $\ket{\mu}$ must appear in the factor $B_0$ as
\begin{equation}\label{eqn:diffactionvarphi0}
	\diff \varphi_0 \ket{\mu} = (e^c_0 +\wun )\ket{\mu} = \ket{\mu+1}+\ket{\mu}.
\end{equation}

\begin{lemma}\label{lem:ModeDecompGaugedLattComplex}
	There is a decomposition of $\diff$-complexes
	\begin{equation}
		B \cong \bigotimes_{n \le 0} B_n.
	\end{equation}
\end{lemma}
\begin{proof}
	Write the basis elements of $B$ as
	\begin{equation}
		\varphi_{-\alpha+1} e^c_{-\beta} \ket{\mu}
		= \brac[\bigg]{\prod_{n \le 0} \varphi_n^{\multp{-\alpha+1}{n}} \brac{e^c_n}^{\multp{-\beta}{n}}} \ket{\mu}.
	\end{equation}
	A basis for $\bigotimes_{n \le 0} B_n$ is given by the tensor product of basis elements for each $B_n$:
	\begin{equation}
		\setbar[\Big]{ \bigotimes_{n \le 0} \varphi_n^{\multp{-\alpha+1}{n}} \brac{e^c_n}^{\multp{-\beta}{n}} \ket{\mu}_n }{ \alpha \in \partitionsFerm, \ \beta \in \partitions }, \quad \text{where} \quad
		\ket{\mu}_n =
		\begin{cases*}
			\ket{\mu} & if $n = 0$,
			\\
			\ket{0} & if $n \ne 0$.
		\end{cases*}
	\end{equation}
	Define the linear map $\rho\colon B \to \bigotimes_{n \le 0} B_n$ so that
	\begin{equation}
		\rho\colon \prod_{n \le 0} \varphi_n^{\multp{-\alpha+1}{n}} \brac{e^c_n}^{\multp{-\beta}{n}} \ket{\mu}
		\mapsto \bigotimes_{n \le 0} \varphi_n^{\multp{-\alpha+1}{n}} \brac{e^c_n}^{\multp{-\beta}{n}} \ket{\mu}_n .
	\end{equation}
	As a bijective map between bases, this extends linearly to a vector space isomorphism.
	Showing that this is a decomposition of $\diff$-complexes is standard, see for example the proof of \zcref{lem:NoSpecFirstDecomp}.
\end{proof}

We now show that the cohomology of each factor $B_n$ is the same as that of the gauged complex (\zcref{prop:GaugedComplexModeCohom}).

\begin{lemma}\label{lem:GaugedLattComplexModeCohom}
	For all $n\le 0$, we have $\cohom{p}{B_n, \diff} \cong \CC \kdelta{p}{0}$.
\end{lemma}
\begin{proof}
	For $n \le -1$, the calculation is identical to the gauged complex of \zcref{prop:GaugedComplexModeCohom}, hence $\cohom{p}{B_n, \diff} \cong \CC \kdelta{p}{0}$.
	For $n=0$, the non-trivial terms are given by
	\begin{equation}
		B^0_0 = \cspn \setbar{ \ket{\mu} }{ \mu \in [ \lambda ] }
    \quad \text{and} \quad
		B^{-1}_0 = \cspn \setbar{ \varphi_0 \ket{\mu} }{ \ r \in \ZZ_{\ge 0}, \ \mu \in [ \lambda ] }.
	\end{equation}
	Since $\diff$ annihilates all the $\ket{\mu}$, we can compute the image $\im \brac{\diff\colon B^{-1}_0 \to B^0_0}$ using \eqref{eqn:diffactionvarphi0} and so we obtain
	\begin{equation}\label{eqn:gaugedLattComplexZeroModes}
		\cohom{0}{B_0, \diff}
		= \frac{\cspn \setbar[\big]{\ket{\mu}}{\mu \in [\lambda]}}{\cspn \setbar[\big]{\ket{\mu+1} + \ket{\mu}}{\mu \in [\lambda]}}
		\simeq \CC. \qedhere
	\end{equation}
\end{proof}

\begin{proposition}{} \label{prop:GaugedComplexLatt}
	The cohomology of the graded differential complex $\brac{B, \diff}$ defined in \eqref{eqn:gaugedLattComplexDef} is concentrated in degree zero and given by $\cohom{p}{B, \diff} \simeq \CC \kdelta{p}{0}$.
\end{proposition}
\begin{proof}
	This follows immediately from applying \tkf\ to the decomposition in \zcref{lem:ModeDecompGaugedLattComplex}, using our result in \zcref{lem:GaugedLattComplexModeCohom}.
\end{proof}

\begin{remark}
	In practice, the gauged lattice complexes that we encounter in this paper have $\ket{\mu} = e^{-a + \mu c} \otimes \bcvac$ for some $[\lambda] \in \CC / \ZZ$.
	Notice that the choice of $\lambda$ does not affect our calculation of the cohomology.
\end{remark}


Next, we consider a filtered spectral sequence associated to a complex we dub the \emph{Cartan complex}.
This complex appears in reduction calculations when dealing with pairs of ghosts $\varphi^*(z)$ and fields $h(z)$ satisfying the commutation relations
\begin{equation}\label{eqn:filteredCartanComplexCommRel}
	\comm{h_m}{\varphi^*_n} = \varphi^*_{m+n} \quad \text{and} \quad \comm{\varphi^*_m}{\varphi^*_n} = 0.
\end{equation}

Consider the differential complex $\brac{B, \diff}$ where $B$ is spanned by the basis
\begin{equation}\label{eqn:filteredCartanComplexDef}
	\setbar[\big]{ h_{-\beta} \varphi^*_{-\alpha} \ket{0} }{ \alpha \in \partitionsFerm, \ \beta \in \partitions }.
\end{equation}
Here $B=\bigoplus_{n\ge 0}B^n$ is again graded by ghost number with $\ghgr \varphi^* = 1$ and $\ghgr h = 0$.
We assume that the differential $\diff$ acts as
\begin{equation}\label{eqn:CartanComplexDiffAction}
	\comm{\diff}{\varphi^*_n} = 0, \quad
	\comm{\diff}{h_n} = \varphi^*_n, \quad
	\diff \ket{0} = 0.
\end{equation}
To compute the cohomology, we introduce an increasing filtration $F^{\bullet} B$ in which $F^p B^n$ is spanned by states of total ghost number $n$ with at most $p$ modes acting on the vacuum $\ket{0}$.
This gives a filtered spectral sequence whose initial page is
\begin{equation}\label{eqn:CartanInitialPage}
	K_0^{p, q} = \frac{ F^p B^{p+q} }{ F^{p-1} B^{p+q} } = \Gr^p B^{p+q},
\end{equation}
where the total degree $n=p+q$ is our ghost number.
Clearly $F^p B^n$ is non-trivial only when $n, p \ge 0$ and $n \le p$, as there are no states with negative ghost grade.
Consequently, the non-trivial terms in our spectral sequence are confined to the fourth quadrant of the $(p, q)$-plane as shown in \zcref{fig:FilteredCartanComplexPageZero}.
\begin{center}
\begin{figure}
	\begin{tikzpicture}[scale=0.85]
		\draw[color=colourBackground1, fill=colourBackground1, opacity=0.4] (4.89, 0) -- (0, 0) -- (0, -0.5) -- (4.89, -5.4) -- cycle;
		\draw[line width=0.5mm, shorten <= 3pt, shorten >= 3pt, gray!10!black!80!] (-1.5, 0) -- (5, 0);
		\draw[line width=0.5mm, shorten <= 3pt, shorten >= 3pt, gray!10!black!80!] (0, 1.5) -- (0, -5);
		\node[anchor=south east] at (0, 1.3) {$q$};
		\node[anchor=north west] at (4.9, 0) {$p$};
		\node[circle, fill=gray!40!white] at (-2, 1) {$K_0^{p, q}$};
		\newcommand{\tempScale}{0.7}
		\foreach [evaluate={\a=int(\x);}] \x in {0, ...,  4} {
                \foreach [evaluate={\b=int(-\y); \c=int(\y-1);}] \y in {0, ...,  4} {
                    \ifthenelse{\a > \c}{
					\filldraw[dotcolour] (\a, \b) circle (2pt);
                    }
                }
            }
        \foreach [evaluate={\a=int(\x);}] \x in {0, ...,  4}{
            \foreach [evaluate={\b=int(-\y); \c=int(\y-2); \d=int(-\y+1);}] \y in {0, ...,  5}{
                \ifthenelse{\a > \c}{
					\draw[draw=arrowcolour, -stealth, shorten <= 3pt, shorten >= 3pt, line width=1pt] (\a, \b) -- (\a, \d);
                    }
                }
            }
        	\foreach [evaluate={\a=int(\x);}] \x in {0, ...,  4} {
                \node[anchor=south west, scale=\tempScale] at (\a, 1) {$0$};
                \filldraw[dotcolour] (\a, 1) circle (2pt);
            }
        \foreach [evaluate={\a=int(\x); \aa=int(\x-1);}] \x in {0, ...,  6} {
			\foreach [evaluate={\b=int(-\y); \c=int(\y+1); \d=int(-\y+1); \dd=int(-\y-1);}] \y in {0, ...,  4}{
				\ifthenelse{\a = \c}{
					\node[anchor=north east, scale=\tempScale] at (\aa, \dd) {$0$};
					\filldraw[dotcolour] (\aa, \dd) circle (2pt);
				}
			}
		}
	\end{tikzpicture}
\caption{The initial page of the filtered Cartan spectral sequence with arrows indicating the action of the differential $\diff_0$.}\label{fig:FilteredCartanComplexPageZero}
\end{figure}
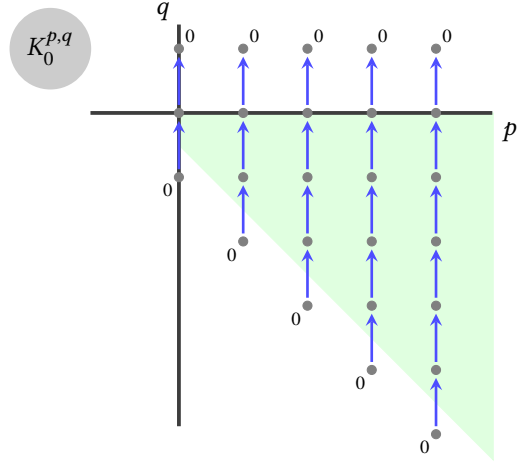
\end{center}

Consider a general basis element $v = h_{-\alpha} \varphi^*_{-\beta} \ket{0}$ of bidegree $\brac[\big]{\lenp{\alpha}+\lenp{\beta}, -\lenp{\alpha}}$.
Using \eqref{eqn:CartanComplexDiffAction}, we describe the action of the differential $\diff_0$ induced on the associated graded space by $\diff$.
\begin{lemma}\label{lem:InitialPageCartanSpecSeq}
	The differential $\diff_0$ acts on a basis state $v$ as
	\begin{equation}\label{eqn:CartanComplexZerothDiff}
		\diff_0 v = \sum^{\lenp{\alpha}}_{i=1} (-1)^{i-1} h_{-\alpha_1} \cdots \omitmode{h_{-\alpha_i}} \varphi^*_{-\alpha_i} \cdots h_{-\alpha_{\lenp{\alpha}}} \varphi^*_{-\beta} \ket{0}.
	\end{equation}
	where $\omitmode{\hphantom{\cdots}}$ indicates omission.
\end{lemma}

\begin{proposition}\label{prop:CartanComplexCohom}
	The cohomology of the Cartan complex $\brac{B, \diff}$ defined in \eqref{eqn:filteredCartanComplexDef} is given by
	\begin{equation}
    \cohom{n}{B, \diff} \cong \CC \kdelta{n}{0}.
	\end{equation}
\end{proposition}
\begin{proof}
	We obtain the subsequent page $K_1^{p, q}$ of the spectral sequence by computing the cohomology of the initial page $\brac{K_0, \diff_0}$.
	It is a straightforward calculation to show that $K_1^{p, q} \cong \CC \kdelta{p}{0} \kdelta{q}{0}$.
	As the page $K_1^{p, q}$ contains only a single non-trivial term, all the $K_R^{p, q}$ with $R \ge 1$ are the same: the spectral sequence has collapsed.
	Since the spectral sequence is bounded, it converges to our desired cohomology.
	We may therefore identify $K_1^{p, q} = K_{\infty}^{p, q}$ with $\cohom{n}{B, \diff}$.
	It follows that $\cohom{n}{B, \diff} \cong \CC \kdelta{n}{0}$.
\end{proof}


\flushleft


\begin{thebibliography}{10}

\bibitem{AdaRea17}
D~Adamovi\'{c}.
\newblock Realizations of simple affine vertex algebras and their modules: the
  cases $\widehat{sl(2)}$ and $\widehat{osp(1,2)}$.
\newblock {\em Comm. Math. Phys.}, 366:1025--1067, 2019.
\newblock \pp{1711.11342}{math.QA}.

\bibitem{ACG21}
D~Adamovi\'c, T~Creutzig and N~Genra.
\newblock Relaxed and logarithmic modules of $\widehat{\mathfrak{sl}_3}$.
\newblock {\em Math. Ann.}, 389:281--324, 2023.
\newblock \pp{2110.15203}{math.RT}.

\bibitem{ACGY20}
D~Adamovi\'c, T~Creutzig, N~Genra and J~Yang.
\newblock The vertex algebras {$\mathcal{R}^{(p)}$} and {$\mathcal{V}^{(p)}$}.
\newblock {\em Comm. Math. Phys.}, 383:1207--1241, 2021.
\newblock \pp{2001.08048}{math.RT}.

\bibitem{AKR21}
D~Adamovi\'c, K~Kawasetsu and D~Ridout.
\newblock A realisation of the {Bershadsky}--{Polyakov} algebras and their
  relaxed modules.
\newblock {\em Lett. Math. Phys.}, 111, 2021.
\newblock \pp{2007.00396}{math.QA}.

\bibitem{AKR23}
D~Adamovi\'c, K~Kawasetsu and D~Ridout.
\newblock Weight module classifications for {Bershadsky}--{Polyakov} algebras.
\newblock {\em Comm. Contemp. Math.}, 26, 2024.
\newblock \pp{2303.03713}{math.QA}.

\bibitem{AM95}
D~Adamovi\'{c} and A~Milas.
\newblock Vertex operator algebras associated to modular invariant
  representations of {$A_1^{\left(1\right)}$}.
\newblock {\em Math. Res. Lett.}, 2:563--575, 1995.
\newblock \opp{9509025}{q-alg}.

\bibitem{AM07}
D~Adamovi\'{c} and A~Milas.
\newblock Logarithmic intertwining operators and {$\mathcal{W}(2,2p-1)$}
  algebras.
\newblock {\em J. Math. Phys.}, 48, 2007.
\newblock \opp{0702081}{math.QA}.

\bibitem{AM09}
D~Adamovi\'{c} and A~Milas.
\newblock Lattice construction of logarithmic modules for certain vertex
  algebras.
\newblock {\em Selecta Math. New Ser.}, 15, 2009.
\newblock \pp{0902.3417}{math.QA}.

\bibitem{AraVan04}
T~Arakawa.
\newblock Vanishing of cohomology associated to quantized {Drinfeld}--{Sokolov}
  reduction.
\newblock {\em Int. Math. Res. Not.}, 2004:729--767, 2004.
\newblock \opp{0303172}{math.QA}.

\bibitem{AraRepW07}
T~Arakawa.
\newblock Representation theory of {W}-algebras.
\newblock {\em Invent. Math.}, 169:219--320, 2007.
\newblock \pp{0506056}{math.QA}.

\bibitem{Ara10}
T~Arakawa.
\newblock Associated varieties of modules over {Kac}--{Moody} algebras and
  {$C_2$}-cofiniteness of {W}-algebras.
\newblock {\em Int. Math. Res. Not.}, 2015:11605--11666, 2015.
\newblock \pp{1004.1554}{math.QA}.

\bibitem{ACK23}
T~Arakawa, T~Creutzig and K~Kawasetsu.
\newblock Weight representations of affine {Kac}--{Moody} algebras and small
  quantum groups.
\newblock {\em Adv. Math.}, 477:110365, 2025.
\newblock \pp{2311.10233}{math.RT}.

\bibitem{AF19}
T~Arakawa and E~Frenkel.
\newblock Quantum {Langlands} duality of representations of
  {$\mathcal{W}$}-algebras.
\newblock {\em Contemporary Mathematics}, 155:2235--2262, 2019.
\newblock \pp{1807.01536}{math.QA}.

\bibitem{AM23}
T~Arakawa and A~Moreau.
\newblock Arc spaces and vertex algebras.
\newblock Available at
  \texttt{https://www.imo.universite-paris-saclay.fr/$\sim$anne.moreau/}.

\bibitem{AraMod16}
T~Arakawa and J~{van~Ekeren}.
\newblock Modularity of relatively rational vertex algebras and fusion rules of
  principal affine {$W$}-algebras.
\newblock {\em Comm. Math. Phys.}, 370:205--247, 2019.
\newblock \pp{1612.09100}{math.RT}.

\bibitem{AraRat19}
T~Arakawa and J~{van~Ekeren}.
\newblock Rationality and fusion rules of exceptional {W}-algebras.
\newblock {\em J. Eur. Math. Soc.}, 25:2763--2813, 2023.
\newblock \pp{1905.11473}{math.RT}.

\bibitem{BelKdV89}
A~Belavin.
\newblock {KdV}-type equations and {W}-algebras.
\newblock {\em Adv. Stud. Pure Math.}, 19:117--125, 1989.

\bibitem{BDT01}
S~Berman, C~Dong and S~Tan.
\newblock Representations of a class of lattice type vertex algebras.
\newblock {\em J. Pure Appl. Algebra}, 176:27--47, 2002.
\newblock \opp{0109215}{math.QA}.

\bibitem{Ber91}
M~Bershadsky.
\newblock Conformal field theories via {Hamiltonian} reduction.
\newblock {\em Comm. Math. Phys.}, 139:71--82, 1991.

\bibitem{BO89}
M~Bershadsky and H~Ooguri.
\newblock Hidden {$SL(n)$} symmetry in conformal field theories.
\newblock {\em Comm. Math. Phys.}, 126:49--83, 1989.

\bibitem{Boa99}
J~Boardman.
\newblock Conditionally convergent spectral sequences.
\newblock {\em Contemporary Mathematics}, 239:49--84, 1999.

\bibitem{BG06}
J~Brundan and S~Goodwin.
\newblock Good grading polytopes.
\newblock {\em Proc. Lon. Math. Soc.}, 94:155--180, 2006.
\newblock \opp{0510205}{math.QA}.

\bibitem{Creu24}
T~Creutzig.
\newblock Resolving {Verlinde's} formula of logarithmic {CFT}.
\newblock \pp{2411.11383}{math.QA}.

\bibitem{CreSL225}
T~Creutzig, J~Fasquel, N~Genra and D~Ridout.
\newblock $\mathfrak{sl}_{2\vert1}$ minimal models {I}: classification of
  irreducible modules.
\newblock In preparation.

\bibitem{CFLN24}
T~Creutzig, J~Fasquel, A~Linshaw and S~Nakatsuka.
\newblock On the structure of {W}-algebras in type {A}.
\newblock {\em Jpn. J. Math.}, 20:1--111, 2025.
\newblock \pp{2403.08212}{math.RT}.

\bibitem{CreUni19}
T~Creutzig, T~Liu, D~Ridout and S~Wood.
\newblock Unitary and non-unitary {$N=2$} minimal models.
\newblock {\em J. High Energy Phys.}, 1906:024, 2019.
\newblock \pp{1902.08370}{math-ph}.

\bibitem{CMY24}
T~Creutzig, R~McRae and J~Yang.
\newblock Ribbon categories of weight modules for affine $\mathfrak{sl}_2$ at
  admissible levels.
\newblock \pp{2411.11386}{math.QA}.

\bibitem{CR12}
T~Creutzig and D~Ridout.
\newblock Modular data and {Verlinde} formulae for fractional level {WZW}
  models {I}.
\newblock {\em Nucl. Phys.}, B865:83--114, 2012.
\newblock \pp{1205.6513}{hep.th}.

\bibitem{CreLog13}
T~Creutzig and D~Ridout.
\newblock Logarithmic conformal field theory: beyond an introduction.
\newblock {\em J. Phys.}, A46:494006, 2013.
\newblock \pp{1303.0847}{hep-th}.

\bibitem{CR13}
T~Creutzig and D~Ridout.
\newblock Modular data and {Verlinde} formulae for fractional level {WZW}
  models {II}.
\newblock {\em Nucl. Phys.}, B875:423--458, 2013.
\newblock \pp{1306.4388}{hep.th}.

\bibitem{EK05}
A~Elashvili and V~Kac.
\newblock Classification of good gradings of simple {Lie} algebras.
\newblock In {\em Lie Groups and Invariant Theory}, volume 213 of {\em Trans.
  Amer. Math. Soc.}, pages 85--104. American Mathematical Society, Providence,
  2005.

\bibitem{FFFN24}
J~Fasquel, Z~Fehily, E~Fursman and S~Nakatsuka.
\newblock Connecting affine {$W$}-algebras: a case study of $\mathfrak{sl}_4$.
\newblock {\em J. Pure Appl. Algebra}, 230:108149, 2026.
\newblock \pp{2408.13785}{math.QA}.

\bibitem{FasQua25}
J~Fasquel and S~Nakatsuka.
\newblock Quantum {Hamiltonian} reductions for {W}-algebras.
\newblock \pp{2512.18743}{math.RT}.

\bibitem{FN23}
J~Fasquel and S~Nakatsuka.
\newblock Orthosymplectic {Feigin}--{Semikhatov} duality.
\newblock {\em Selecta Math. New Ser.}, 31:69, 2025.
\newblock \pp{2307.14574}{math.RT}.

\bibitem{FRR24}
J~Fasquel, C~Raymond and D~Ridout.
\newblock Modularity of admissible-level $\mathfrak{sl}_3$ minimal models with
  denominator $2$.
\newblock {\em Comm. Math. Phys.}, 406:279, 2025.
\newblock \pp{2406.10646}{math.QA}.

\bibitem{FehHook23}
Z~Fehily.
\newblock Inverse reduction for hook-type {W}-algebras.
\newblock {\em Comm. Math. Phys.}, 405, 2023.
\newblock \pp{2306.14673}{math.QA}.

\bibitem{FehSub23}
Z~Fehily.
\newblock Subregular {W}-algebras of type {$A$}.
\newblock {\em Comm. Contemp. Math.}, 25, 2023.
\newblock \pp{2111.05536}{math.QA}.

\bibitem{FehPri25}
Z~Fehily, C~Raymond and D~Ridout.
\newblock The principal {W}-algebra of $\mathfrak{psl}_{2\vert2}$.
\newblock {\em Symmetry Integrability Geom. Methods Appl.}, 2026 (to appear).
\newblock \pp{2509.04795}{math.QA}.

\bibitem{FR22}
Z~Fehily and D~Ridout.
\newblock Modularity of {Bershadsky}--{Polyakov} minimal models.
\newblock {\em Lett. Math. Phys.}, 112, 2022.
\newblock \pp{2110.10336}{math.QA}.

\bibitem{FF90}
B~Feigin and E~Frenkel.
\newblock Quantization of the {Drinfeld}--{Sokolov} reduction.
\newblock {\em Phys. Lett.}, B246:75--81, 1990.

\bibitem{FST98}
B~Feigin, A~Semikhatov and I~Tipunin.
\newblock Equivalence between chain categories of representations of affine
  $\mathfrak{sl}(2)$ and {$N=2$} superconformal algebras.
\newblock {\em J. Math. Phys.}, 39:3865--3905, 1998.
\newblock \opp{9701043}{hep-th}.

\bibitem{FFHST02}
J~Fjelstad, J~Fuchs, J.S Hwang, A.M Semikhatov and I.Yu Tipunin.
\newblock Logarithmic conformal field theories via logarithmic deformations.
\newblock {\em Nucl. Phys.}, B 633:379--413, 2002.

\bibitem{FMS86}
D~Friedan, E~Martinec and S~Shenker.
\newblock Conformal invariance, supersymmetry and string theory.
\newblock {\em Nucl. Phys.}, B271:93--165, 1986.

\bibitem{FutCla01}
V~Futorny and A~Tsylke.
\newblock Classification of irreducible nonzero level modules with
  finite-dimensional weight spaces for affine {Lie} algebras.
\newblock {\em J. Algebra}, 238:426--441, 2001.

\bibitem{GabFus01}
M~Gaberdiel.
\newblock Fusion rules and logarithmic representations of a {WZW} model at
  fractional level.
\newblock {\em Nucl. Phys.}, B618:407--436, 2001.
\newblock \opp{0105046}{hep-th}.

\bibitem{GawNon91}
K~Gaw\c{e}dzki.
\newblock Noncompact {WZW} conformal field theories.
\newblock {\em NATO Science Series II: Mathematics, Physics and Chemistry},
  295:247--274, 1992.
\newblock \opp{9110076}{hep-th}.

\bibitem{GenRed25}
N~Genra and T~Juillard.
\newblock Reduction by stages for affine {W}-algebras.
\newblock \pp{2501.04501}{math.RT}.

\bibitem{KRW03}
V~Kac, S~Roan and M~Wakimoto.
\newblock Quantum reduction for affine superalgebras.
\newblock {\em Comm. Math. Phys.}, 241:307--342, 2003.
\newblock \opp{0302015}{math-ph}.

\bibitem{KW04}
V~Kac and M~Wakimoto.
\newblock Quantum reduction and representation theory of superconformal
  algebras.
\newblock {\em Adv. Math.}, 185:400--458, 2004.
\newblock \opp{0304011}{math-ph}.

\bibitem{KR19}
K~Kawasetsu and D~Ridout.
\newblock Relaxed highest-weight modules {I}: Rank 1 cases.
\newblock {\em Comm. Math. Phys.}, 368:627--663, 2019.
\newblock \pp{1803.01989}{math.RT}.

\bibitem{KR21}
K~Kawasetsu and D~Ridout.
\newblock Relaxed highest-weight modules {II}: classifications for affine
  vertex algebras.
\newblock {\em Comm. Contemp. Math.}, 24:2150037, 2022.
\newblock \pp{1906.02935}{math.RT}.

\bibitem{KawAdm21}
K~Kawasetsu, D~Ridout and S~Wood.
\newblock An admissible-level $\mathfrak{sl}_3$ model.
\newblock {\em Lett. Math. Phys.}, 112:96, 2022.
\newblock \pp{2107.13204}{math.QA}.

\bibitem{KohFus88}
I~Koh and P~Sorba.
\newblock Fusion rules and (sub)modular invariant partition functions in
  nonunitary theories.
\newblock {\em Phys. Lett.}, B215:723--729, 1988.

\bibitem{Li94Sym}
H~Li.
\newblock Symmetric invariant bilinear forms on vertex operator algebras.
\newblock {\em J. Pure Appl. Algebra}, 96:279--297, 1994.

\bibitem{Li97}
H~Li.
\newblock The physics superselection principle in vertex operator algebra
  theory.
\newblock {\em J. Algebra}, 196:436--457, 1997.

\bibitem{Li05}
H~Li.
\newblock Abelianizing vertex algebras.
\newblock {\em Comm. Math. Phys.}, 259:391--411, 2005.
\newblock \opp{0409140}{math.QA}.

\bibitem{MR97}
J~Madsen and E~Ragoucy.
\newblock Secondary quantum {Hamiltonian} reductions.
\newblock {\em Comm. Math. Phys.}, 185:509--541, 1997.
\newblock \opp{9503042}{hep-th}.

\bibitem{MalStr00}
J~Maldacena and H~Ooguri.
\newblock Strings in {$AdS_3$} and the {$\mathrm{SL} \left( 2 , R \right)$}
  {WZW} model. {Part} 1: The spectrum.
\newblock {\em J. Math. Phys.}, 42:2929--2960, 2001.
\newblock \opp{0001053}{hep-th}.

\bibitem{McRRat21}
R~McRae.
\newblock On rationality for {$C_2$}-cofinite vertex operator algebras.
\newblock \pp{2108.01898}{math.QA}.

\bibitem{NHCW24}
H~Nakano, F~{Orosz Hunziker}, A~{Ros Camacho} and S~Wood.
\newblock Fusion rules and rigidity for weight modules over the simple
  admissible affine $\mathfrak{sl} \left( 2 \right)$ and $\mathcal{N}=2$
  superconformal vertex operator superalgebras.
\newblock \pp{2411.11387}{math.QA}.

\bibitem{PolBos81}
A~Polyakov.
\newblock Quantum geometry of bosonic strings.
\newblock {\em Phys. Lett. B}, 103:207--210, 1981.

\bibitem{PolFer81}
A~Polyakov.
\newblock Quantum geometry of fermionic strings.
\newblock {\em Phys. Lett. B}, 103:211--213, 1981.

\bibitem{Pol90}
A~Polyakov.
\newblock Gauge transformations and diffeomorphisms.
\newblock {\em Int. J. Mod. Phys.}, A5:833--842, 1990.

\bibitem{Rid09}
D~Ridout.
\newblock $\mathfrak{sl}(2)_{-1/2}$: A case study.
\newblock {\em Nucl. Phys.}, B814:485--521, 2009.
\newblock \pp{0810.3532}{hep-th}.

\bibitem{RidFus10}
D~Ridout.
\newblock Fusion in fractional level $\widehat{\mathfrak{sl}}(2)$-theories with
  $k=-\tfrac{1}{2}$.
\newblock {\em Nucl. Phys.}, B848:216--250, 2011.
\newblock \pp{1012.2905}{hep-th}.

\bibitem{RidBos14}
D~Ridout and S~Wood.
\newblock Bosonic ghosts at $c=2$ as a logarithmic {CFT}.
\newblock {\em Lett. Math. Phys.}, 105:279--307, 2015.
\newblock \pp{1408.4185}{hep-th}.

\bibitem{RidJac14}
D~Ridout and S~Wood.
\newblock From {Jack} polynomials to minimal model spectra.
\newblock {\em J. Phys.}, A48:045201, 2015.
\newblock \pp{1409.4847}{hep-th}.

\bibitem{RW15}
D~Ridout and S~Wood.
\newblock Relaxed singular vectors, {Jack} symmetric functions and fractional
  level $\widehat{\mathfrak{sl}}(2)$ models.
\newblock {\em Nucl. Phys.}, B894:621--664, 2015.
\newblock \pp{1501.07318}{hep-th}.

\bibitem{RidVer14}
D~Ridout and S~Wood.
\newblock The {Verlinde} formula in logarithmic {CFT}.
\newblock {\em J. Phys. Conf. Ser.}, 597:012065, 2015.
\newblock \pp{1409.0670}{hep-th}.

\bibitem{SS87}
A~Schwimmer and N~Seiberg.
\newblock Comments on the {$N = 2,3,4$} superconformal algebras in two
  dimensions.
\newblock {\em Phys. Lett. B}, 184:191--196, 1987.

\bibitem{Sem94}
A~Semikhatov.
\newblock Inverting the {Hamiltonian} reduction in string theory.
\newblock In {\em 28th International Symposium on Particle Theory,
  Wendisch-Rietz, Germany}, pages 156--167, 1994.
\newblock \opp{9410109}{hep-th}.

\bibitem{SemMFF93}
A~Semikhatov.
\newblock The {MFF} singular vectors in topological conformal theories.
\newblock {\em Modern Phys. Lett.}, A9:1867--1896, 1994.
\newblock \opp{9311180}{hep-th}.

\bibitem{Sim25}
D~Simon.
\newblock Representation theory of the principal equivariant affine
  {$\mathcal{W}$}-algebra and {Langlands} duality.
\newblock \pp{2510.06990}{math.RT}.

\bibitem{Siu19}
S~Siu.
\newblock {\em Singular vectors for the {$\mathcal{W}_N$} algebras and the
  {BRST} cohomology for relaxed highest-weight
  {$L_k(\mathfrak{sl}_2)$}-modules}.
\newblock PhD thesis, University of Melbourne, 2019.

\bibitem{TesStr97}
J~Teschner.
\newblock On structure constants and fusion rules in the
  {$SL(2,\mathbb{C})/SU(2)$} {WZNW} model.
\newblock {\em Nucl. Phys.}, B546:390--422, 1999.
\newblock \opp{9712256}{hep-th}.

\bibitem{Wan93}
W~Wang.
\newblock Rationality of {Virasoro} vertex operator algebras.
\newblock {\em Int. Math. Res. Not.}, 1993:197--211, 1993.

\bibitem{Wei94}
C~Weibel.
\newblock {\em An Introduction to Homological Algebra}, volume~38 of {\em
  Cambridge Studies in Advanced Mathematics}.
\newblock Cambridge University Press, Cambridge, 1994.

\end{thebibliography}
\providecommand{\opp}[2]{\textsf{arXiv:\mbox{#2}/#1}}
\providecommand{\pp}[2]{\textsf{arXiv:#1 [\mbox{#2}]}}

\end{document}